%% file: paper.tex
\documentclass{mcom-l}

\usepackage{graphicx}
\usepackage{a4wide}
\usepackage{amstext, amsmath, amssymb, amsthm}
\usepackage{latexsym}
\usepackage{booktabs}
\usepackage{url}
\usepackage{import}
\usepackage{fancyvrb}
\usepackage{stmaryrd}
\usepackage{pdfsync}
\usepackage{algorithm}
\usepackage{tikz}
\usepackage{pgfplots}
\usepackage{pgfplotstable}
\usepackage{tabularx}
\usepackage{todonotes}
\usepackage[numbers,square,sort&compress]{natbib}
\usepackage{caption}
\usepackage{subcaption}
\usepackage{enumitem}

\newcommand{\RR}{\mathbb{R}}

\newcommand{\bfI}{\boldsymbol I}

\newcommand{\bfzero}{\boldsymbol 0}

\newcommand{\mcV}{\mathcal{V}}
\newcommand{\mcK}{\mathcal{K}}

\newcommand{\mcN}{\mathcal{N}}
\newcommand{\mcF}{\mathcal{F}}
\newcommand{\mcA}{\mathcal{A}}

\newcommand{\mcT}{\mathcal{T}}
\newcommand{\mcH}{\mathcal{H}}

\newcommand{\mcX}{\mathcal{X}}
\newcommand{\mcS}{\mathcal{S}}

\newcommand{\Cyl}{\text{Cyl}}
\newcommand{\Cone}{\text{Cone}}
\newcommand{\barGamma}{\overline{\Gamma}}
\newcommand{\barn}{\overline{n}}
\newcommand{\barp}{\overline{p}}
\newcommand{\barQ}{\overline{Q}}

\newcommand{\barrho}{\overline{\rho}}

\newcommand{\mcTh}{{\mcT_h}}

\newcommand{\jump}[1]{[#1]}
\newcommand{\scp}[1]{\langle {#1} \rangle}
\newcommand{\mean}[1]{\langle {#1} \rangle}

\newcommand{\Gammah}{{\Gamma_h}}
\newcommand{\nablas}{\nabla_\Gamma}
\newcommand{\nablash}{\nabla_{\Gamma_h}}

\newcommand{\Ps}{{P}_\Gamma}
\newcommand{\Qs}{{Q}_\Gamma}
\newcommand{\Psh}{{P}_{\Gamma_h}}
\newcommand{\Qsh}{{Q}_{\Gamma_h}}

\newcommand{\foralls}{\forall\,}
\newcommand{\dGamma}{\,\mathrm{d} \Gamma}
\newcommand{\dGammah}{\,\mathrm{d} \Gamma_h}

\newcommand{\dy}{\,\mathrm{d}y}
\newcommand{\ds}{\,\mathrm{d}s}

\newcommand{\IR}{\mathbb{R}}

\DeclareMathOperator{\spann}{span}
\DeclareMathOperator{\dist}{dist}
\DeclareMathOperator{\diam}{diam}
\DeclareMathOperator{\sig}{sig}

\newcommand{\onehalf}{\frac{1}{2}}

\usepackage[plainpages=false]{hyperref}
\hypersetup{
  bookmarksopen=true,
                bookmarksnumbered=true,
                pdfborder={0 0 0},
                pdftitle= {Cut Finite Element Methods for Partial Differential Equations on Embedded Manifolds of Arbitrary Codimensions},
                pdfauthor={E. Burman, P. Hansbo, M.G. Larson, A. Massing},
                pdfcreator={E. Burman, P. Hansbo, M.G. Larson, A. Massing},
                colorlinks=true,
                linkcolor=blue,
                urlcolor=blue
        }

\usepackage{verbatim} 
\DefineVerbatimEnvironment{code}{Verbatim}{frame=single,rulecolor=\color{blue}}

        \newtheorem{theorem}{Theorem}[section]
\newtheorem{proposition}[theorem]{Proposition}
\newtheorem{lemma}[theorem]{Lemma}

\theoremstyle{remark}
\newtheorem{remark}[theorem]{Remark}

\theoremstyle{definition}

\numberwithin{equation}{section}

\begin{document}
\title[CutFEMs for
PDEs on Embedded Submanifolds]{\bf Cut Finite Element Methods for
Partial Differential Equations on Embedded Manifolds of Arbitrary
Codimensions}


\author[E. Burman]{Erik Burman}
\address[Erik Burman]{Department of Mathematics, University College London, London, UK--WC1E 6BT, United Kingdom}
\email{e.burman@ucl.ac.uk}

\author[P. Hanbso]{Peter Hansbo}
\address[Peter Hansbo]{Department of Mechanical Engineering, J\"onk\"oping University,
SE-55111 J\"onk\"oping, Sweden.}
\email{peter.hansbo@ju.se}

\author[M. G. Larson]{Mats G.\ Larson}
\author[A. Massing]{Andr\'e Massing}
\address[Mats G.\ Larson, Andr\'e Massing]{Department of Mathematics and Mathematical Statistics, Ume{\aa} University, SE-90187 Ume{\aa}, Sweden}
\email{mats.larson@umu.se}
\email{andre.massing@umu.se}

\date{\today}


\maketitle

\begin{abstract}
  We develop a theoretical framework for the analysis of
  stabilized cut finite element methods for the Laplace-Beltrami
  operator on a manifold embedded in $\IR^d$ of arbitrary codimension.
  The method is based on using continuous piecewise polynomials on a
  background mesh in the embedding space for approximation together with
  a stabilizing form that ensures that the resulting problem is
  stable.
  The discrete manifold is represented using a
  triangulation which does not match the background mesh
  and does not need to be shape-regular, which includes
  level set descriptions of codimension one manifolds and
  the non-matching embedding of independently triangulated manifolds
  as special cases.
  We identify abstract key assumptions
  on the stabilizing form which allow us to prove a bound on the condition number of the
  stiffness matrix and optimal order a priori estimates.
  The key assumptions are verified for
  three different realizations of the stabilizing
  form including a novel stabilization approach based
  on penalizing the surface normal gradient on the background mesh.
  Finally, we present numerical results illustrating our results
  for a curve and a surface embedded in $\IR^3$.
\end{abstract}

\section{Introduction}
In this paper, we develop a unified analysis for stabilized cut finite element methods
for the Laplace-Beltrami problem
\begin{align}
  \label{eq:LB}
  -\Delta_\Gamma u = f \quad \text{on $\Gamma$}
\end{align}
posed on compact, smooth $d$-dimensional manifolds
embedded in $\RR^k$ without boundary.
Here,
$f\in L^2(\Gamma)$ is assumed to satisfy $\int_\Gamma f = 0$
to guarantee the unique solvability of problem~(\ref{eq:LB}).

The cut finite element method is a recently developed, general unfitted finite
element technique to facilitate the numerical solution of partial differential
equations (PDEs) on complex geometries~\cite{BurmanClausHansboEtAl2014},
including PDEs on surfaces embedded in
$\RR^3$~\cite{OlshanskiiReuskenGrande2009}. The method uses restrictions of
standard continuous piecewise linear finite elements defined on a partition of
the embedding space into tetrahedra, the so called background mesh, to a
piecewise linear approximation of the exact surface. The discrete surface is
allowed to cut through the background mesh in an arbitrary fashion. The set of
elements in the background mesh intersected by the discrete surface forms the so
called active mesh which supports the piecewise linears involved in the
computation. This approach leads to a potentially ill posed stiffness matrix and
therefore either preconditioning \cite{OlshanskiiReusken2014} or stabilization
\cite{BurmanHansboLarson2015, BurmanHansboLarsonEtAl2016c} is required. In
\cite{BurmanHansboLarson2015}, a consistent stabilization term was introduced
that provides control of the jump in the normal gradient on each of the interior
faces in the active mesh. In constrast, the stabilization proposed and analyzed
in \cite{BurmanHansboLarsonEtAl2016c} provides control over the full gradient on
the elements in the active mesh and is (weakly) inconsistent. Optimal order a
priori error estimates and condition number estimates are established for both
types of stabilization in \cite{BurmanHansboLarson2015,
BurmanHansboLarsonEtAl2016c}. Further developments in this area include
convection problems on surfaces
\cite{OlshanskiiReuskenXu2014,BurmanHansboLarsonEtAl2015b}, cut discontinuous
Galerkin methods~\cite{BurmanHansboLarsonEtAl2016a}, adaptive methods based on a
posterior error
estimates~\cite{ChernyshenkoOlshanskii2015,DemlowOlshanskii2012}, coupled
surface bulk problems
\cite{BurmanHansboLarsonEtAl2014,GrossOlshanskiiReusken2014}, and time dependent
problems
\cite{OlshanskiiReuskenXu2014a,OlshanskiiReusken2014,HansboLarsonZahedi2015b}.
See also the review article \cite{BurmanClausHansboEtAl2014} on cut finite
element methods and references therein, and \cite{DziukElliott2013} for general
background on finite element methods for surface partial differential equations.

\subsection{Novel contributions}
In this contribution we develop a stabilized cut finite element framework for the
Laplace-Beltrami operator on a $d$-dimensional manifold without
boundary embedded in $\IR^k$, with arbitrary codimension $1\leq k-d \leq k-1$.
Common examples includes curves and surfaces embedded in $\IR^3$, but
our results cover the general situation.
We develop a general theoretical framework for proving a priori error estimates
and bounds on the condition number that relies on abstract properties
on the forms involved in the problem.  In particular, only certain
properties of the stabilization form are required.

We study three different stabilizing forms, one based on the jump in
the normal gradient across faces in the active mesh, one based on the
full gradient on the active mesh of simplices, and finally, a new
stabilizing form based on the surface normal gradient on the active mesh.
The
consistency error in the normal gradient stabilization is much smaller
compared to the full gradient stabilization and therefore more
flexible scalings of the stabilization term are possible.
Additionally, it works for higher order approximations,
which we comment on as well.
We verify that the assumptions on the stabilizing forms in the
abstract framework are satisfied for all three stabilizing forms.
In the case of surfaces embedded in $\IR^3$, the face and normal
gradient stabilization terms were individually studied in the
aforementioned references \cite{BurmanHansboLarson2015} and
\cite{BurmanHansboLarsonEtAl2016c}.

The geometric estimates required to bound the
consistency error in the abstract setting
are established for the full range of possible codimensions.
We start from a very general setting where the discrete manifold is
described by a ---possibly irregular--- triangulation, which is not required to match the
background mesh, and which satisfies certain approximation assumptions.
Then we derive all the necessary geometry approximation estimates.
The absence of a compact and explicit description of the
closest point projection and its derivative in the case of codimensions $k - d >
1$ demands an alternative and more general route to establish the necessary
geometric estimates.

We prove interpolation error estimates in this general setting, which
extends previous results to higher codimensions. In the case of higher
codimensions special care is necessary in the derivation of the
interpolation estimate. We note that the standard case of a
codimension one hypersurface in $\IR^k$ described as a levelset of a
piecewise linear continuous function is contained as a special case of
our analysis.

\subsection{Outline} 
The outline of the paper is as follows: 
We start with recalling the weak formulation of the continuous Laplace-Beltrami problem
in Section~\ref{sec:model-problem}.
Then we formulate the
abstract stabilized cut finite element
framework in Section~\ref{sec:cutfem},
followed by a number of realizations based on
triangulated discrete manifolds and three different stabilizing forms in
Section~\ref{sec:cutfem-realizations}.
The abstract estimates for the
condition number and the a priori error are presented in
Section~\ref{sec:abstract-condition-number-estimates} and
Section~\ref{sec:abstract-a-priori-analysis}, respectively,
leading us to a few key assumptions on the stabilization forms.
To prepare the verification of these key assumptions, 
geometric estimates involving the gradient of lifted and
extended functions and the change of the domain integration are established in
Section~\ref{sec:geometric-estimates}. In the same Section we also introduce the concept 
of ``fat intersections'' between the discrete manifold and the underlying background mesh.
The subsequent Section~\ref{sec:poincare-estimate-verification} is
devoted to verify the abstract assumptions for the realizations of the
stabilization forms.  
A suitable interpolation operator is
presented in Section~\ref{sec:interpolation-properties}, together with a proof
of the interpolation error estimates.  
The theoretical development is concluded
by the verification of the quadrature and consistency properties of the
bilinear forms given in Section~\ref{sec:quadrature-consistency-error}.
Finally, in Section~\ref{sec:numerical_results}, we present illustrating
numerical examples for surfaces and curves embedded in $\IR^3$ that corroborate
our theoretical findings.

\section{Weak Formulation of the Continuous Model Problem}
\label{sec:model-problem}
\subsection{The Continuous Manifold}
\label{ssec:preliminaries}
In what follows, $\Gamma$ denotes a boundaryless smooth compact
manifold of dimension $d$ which is embedded in ${{\RR}}^{k}$
and thus has codimension $c = k - d$.
For each $p \in \Gamma$, we denote by $N_p \Gamma$
the orthogonal complement of the tangent space $T_p \Gamma$
in $\RR^k$,
\begin{align}
  N_p \Gamma = (T_p \Gamma)^{\bot} = \{ v \in \RR^k : \langle v,w \rangle = 0\; \foralls w \in T_p \Gamma \}.
  \label{eq:def-normal-space}
\end{align}
Then the normal bundle $N \Gamma$ of $\Gamma$ is defined as the
collection of all normal spaces $N_p \Gamma$; that is
\begin{align}
  N \Gamma = \{ (p,v) \in \Gamma \times \RR^k : v \in N_p \Gamma \}.
  \label{eq:def-normal-bundle}
\end{align}
Locally, the manifold  $\Gamma$ can be
equipped with a smooth adapted moving orthonormal frame
$ \Gamma \ni p \mapsto \{e_i(p)\}_{i=1}^{k} =  \{t_i(p)\}_{i=1}^{d} \cup \{n_i(p)\}_{i=1}^{c}$
where
$\{t_i(p)\}_{i=1}^{d}$ and
$\{n_i(p)\}_{i=1}^{c}$ is a orthonormal
basis of $T_p \Gamma$ and $N_p \Gamma$, respectively.
With the help of such an adapted orthonormal frame,  the
orthogonal projection
$\Ps$ and $\Qs$
of $\RR^{k}$ onto the, respectively, tangent and normal spaces of $\Gamma$ at $x \in \Gamma$
are given by
\begin{equation}
  \Ps =  \sum_{i=1}^d t_i \otimes t_i, \qquad
  \Qs =  \sum_{i=1}^{c} n_i \otimes n_i, \qquad
  \label{eq:gamma-projectors}
\end{equation}
or equivalently, since $\RR^k = T_p \Gamma \oplus N_p \Gamma \; \foralls p \in
\Gamma$, by
\begin{equation}
  \Qs = I -  \sum_{i=1}^d t_i \otimes t_i, \qquad
  \Ps = I - \sum_{i=1}^{c} n_i \otimes n_i, \qquad
  \label{eq:gamma-projectors-complimentary}
\end{equation}
where $I$ is the identity matrix.

The normal bundle $N\Gamma$ can be used to define adapted coordinates in the
$\delta$ tubular neighborhood $U_{\delta}(\Gamma) = \{ x \in \RR^k : \rho(x) < \delta\}$ where
$\rho(x) = \dist(x,\Gamma)$ is the distance function associated to $\Gamma$.
Introducing the set
\begin{align}
  N_{\delta} \Gamma = \{ (p,v) \in N \Gamma : \|v\|_{\RR^k} < \delta \},
  \label{eq:def-normal-bundle-delta}
\end{align}
it is well-known that for a smooth, compact embedded manifold without boundary, the mapping
 \begin{align}
   \Psi: N_\delta \Gamma  \ni (p,v) \mapsto p + v \in U_{\delta}(\Gamma)
   \label{eq:def-Psi}
 \end{align}
in fact defines a diffeomorphism if $0 < \delta \leqslant \delta_0$ for some $\delta_0$ small enough,
see, e.g., ~\cite[p.93]{Bredon1993} for a proof.
Assuming from now on that $\delta \leqslant \delta_0$, the closest point
projection which maps $x\in U_{\delta}(\Gamma)$ to its uniquely defined
closest point on $\Gamma$
is given by the smooth retraction
\begin{align}
  p:  U_{\delta}(\Gamma) \ni x \mapsto  \Pi_{k} \Psi^{-1}(x) \in \Gamma,
  \label{eq:closest-point-projection}
\end{align}
where $\Pi_k: N_{\Gamma} \ni (q,v) \mapsto q$ is the canonical projection of the normal bundle $N \Gamma$ to its
base manifold $\Gamma$.
The closest point projection allows to extend any function on $\Gamma$ to its
tubular neighborhood $U_{\delta}(\Gamma)$ using the pull back
\begin{equation}
  \label{eq:extension}
  u^e(x) = u \circ p (x).
\end{equation}
On the other hand,
for any subset $\widetilde{\Gamma} \subseteq U_{\delta}(\Gamma)$ such that
$p: \widetilde{\Gamma} \to \Gamma $ defines a bijective mapping, a function $w$ on
$\widetilde{\Gamma}$ can be lifted to $\Gamma$ by the push forward defined by
\begin{align}
  (w^l)(p(x)) = w(x).
\end{align}
Finally, for any function space $V$ defined on $\Gamma$,
the space consisting of extended functions is denoted by $V^e$ and
correspondingly, the notation $V^l$ refers to the lift of a
function space~$V$ defined on $\widetilde{\Gamma}$.

\subsection{The Continuous Weak Problem}
A function $u: \Gamma \to \RR$ is of class $C^l(\Gamma)$
if there exists an extension $\overline{u} \in C^l(U)$ with $\overline{u}|_{\Gamma} = u$
for some $k$-dimensional neighborhood $U$ of $\Gamma$.
Using the tangential projection $\Ps$,
the tangent gradient operator $\nabla_\Gamma$ on $\Gamma$ can be defined
by
\begin{equation}
  \nablas u = \Ps \nabla \overline{u},
  \label{eq:tangent-gradient}
\end{equation}
with $\nabla$ being the full ${{\RR}}^{k}$ gradient.
It can easily be shown that the definition~\eqref{eq:tangent-gradient}
is independent of the extension $\overline{u}$. Then the
Laplace-Beltrami operator~$\Delta_\Gamma$ on $\Gamma$ is defined as
\begin{equation}
  \Delta_\Gamma = \nabla_\Gamma \cdot \nabla_\Gamma
\end{equation}
and the corresponding weak formulation of problem~(\ref{eq:LB}) is to seek
$u \in H^1(\Gamma)/\RR$
such that
\begin{equation}
  a(u,v) = l(v) \quad \forall v \in H^1(\Gamma)/\RR,
  \label{eq:cont-variational-form}
\end{equation}
where
\begin{equation}
  a(u,v) = (\nabla_\Gamma u, \nabla_\Gamma v)_\Gamma, \qquad l(v) = (f,v)_\Gamma,
\end{equation}
and $(v,w)_\Gamma = \int_\Gamma v w$ is the $L^2$ inner product.
We let $\| w \|^2_\Gamma = (w,w)_\Gamma $ denote the $L^2(\Gamma)$ norm on $\Gamma$ and
introduce the Sobolev $H^m(\Gamma)$ space as the subset of
$L^2$ functions for which the norm
\begin{equation}
  \| w \|^2_{m,\Gamma} = \sum_{l=0}^m \| D^{l}_\Gamma w\|_\Gamma^2,
  \quad m = 0,1,2
\end{equation}
is defined. Here, the $L^2$ norm for a matrix is based on the pointwise
Frobenius norm, $D^{0}_\Gamma w = w$ and the derivatives $D^{1}_\Gamma = \Ps\nabla w,
D^{2}_\Gamma w = \Ps(\nabla \otimes \nabla w)\Ps$ are taken in a weak sense.

It follows from the Lax-Milgram lemma that the
problem~\eqref{eq:cont-variational-form} has a unique
solution. For smooth surfaces we also have the elliptic regularity
estimate
\begin{equation}
  \| u \|_{2,\Gamma} \lesssim \|f\|_\Gamma.
  \label{eq:ellreg}
\end{equation}
Here and throughout the paper we employ the notation $\lesssim$ to denote less
or equal up to a positive constant that is always independent of the
mesh size. The binary relations $\gtrsim$ and $\sim$ are defined analogously.

\section{The Abstract Cut Finite Element Formulation}
\label{sec:cutfem}
The cut finite element formulation for the numerical solution
of \eqref{eq:cont-variational-form} is based on two ingredients.
First, a geometric approximation $\Gamma_h$ of the embedded manifold $\Gamma$
has be to provided which facilitates the numerical computation
of the discrete counterpart of~\eqref{eq:cont-variational-form}.
Second, a discretization of (a neighborhood) of the embedding space $\RR^k$
is required to provide the approximation spaces in which the numerical solution
will be sought.
We start with specifying the requirements for $\Gamma_h$.

\subsection{The Discrete Manifold}
\label{ssec:domain-disc}
For $\Gamma$, let $\mcK_h = \{ K\} $ be a
conform mesh without boundary consisting of
$d$ dimensional simplices $K$.
While we do not require that the
simplices are shape-regular, we assume that they are
non-degenerated.
On $\Gamma_h = \bigcup_{K \in \mcK_h} K$ we can then define
the piecewise constant, discrete tangential projection $\Psh$
as the orthogonal projection on the $d$-dimensional (affine) subspace defined by each
$K\in \mcK_h$.
We assume that:
\begin{itemize}
  \item $\Gamma_h \subset U_{\delta_0}(\Gamma)$ and that the closest
    point mapping $p:\Gamma_h \rightarrow \Gamma$ is a bijection for
    $0 < h \leq h_0$.
  \item The following estimates hold
    \begin{equation} \
      \| \rho \|_{L^\infty(\Gamma_h)} \lesssim h^2, \qquad
      \| \Ps^e - \Psh \|_{L^\infty(\Gamma_h)} \lesssim h.
      \label{eq:geometric-assumptions-II}
    \end{equation}
\end{itemize}

Before we proceed, we observe that the assumption on the convergence of the tangential projectors
can be reformulated as follows:
\begin{lemma}
  The following assumptions are equivalent.
  \begin{enumerate}
    \item There exists a piecewise smooth moving orthonormal tangential frame $\{t_1^h\}_{i=1}^d$ such
      that
      \begin{align}
        \| t_i^e - t_i^h \|_{L^{\infty}(\Gamma_h)} \lesssim h   \quad   \text{for } i = 1,\ldots, d.
        \label{eq:tangential-projector-frame}
      \end{align}
    \item The discrete tangential projection $\Psh$ satisfies
      \begin{align}
        \| \Ps^e - \Psh \|_{L^{\infty}(\Gamma_h)} \lesssim h.
        \label{eq:tangential-projector-est}
      \end{align}
    \item There exists a piecewise smooth moving orthonormal normal frame $\{n_i^h\}_{i=1}^{c}$ such
      that
      \begin{align}
        \| n_i^e - n_i^h \|_{L^{\infty}(\Gamma_h)} \lesssim h \quad   \text{for } i = 1,\ldots, c.
        \label{eq:normal-frame-est}
      \end{align}
    \item The discrete normal projection $\Qsh$ satisfies
      \begin{align}
        \| \Qs^e - \Qsh \|_{L^{\infty}(\Gamma_h)} \lesssim h.
        \label{eq:normal-projector-est}
      \end{align}
  \end{enumerate}
  \begin{proof}
    The proof is elementary and included only for completeness. Clearly,
    with
    \begin{align}
      \Psh = \sum_{i=1}^{d} t_i^h \otimes t_i^h,
      \qquad
      \Qsh = \sum_{i=1}^{c} n_i^h \otimes n_i^h,
      \label{eq:gammah-projectors}
    \end{align}
    (1) $\Rightarrow$ (2) and (3) $\Rightarrow$ (4). Moreover,
    since $I = \Psh + \Qsh$, the assumptions (2) and (4) are
    clearly equivalent. It remains to show that existence of either the
    discrete normal or tangential projector with the assumed convergence
    properties implies the existence of a corresponding discrete normal
    and tangential frame. For any $K\in \widetilde{\mcK}_h$, we take a
    smooth moving frame $\{E_i\}_{i=1}^k = \{t_i\}_{i=1}^d \cup \{n_i\}_{i=1}^{c}$
    on $K^l = p(K)$. Set $\widetilde{t}_i^h = \Psh(t_i^e) \in T(K)$
    and $\widetilde{n}_i^h = \Qsh(n_i^e) \in N(K)$ and observe that
    $
    \| t_i^e - \widetilde{t}_i^h \|_{L^{\infty}(\Gammah)}
    = \| (\Ps^e - \Psh)(t_i^e) \|_{L^{\infty}(\Gammah)} \lesssim h
    $
    and similar for $\widetilde{n}_i^h$. Consequently,
    $\{\widetilde{E}_i^h\}$
    defines an almost orthonormal basis
    of $\RR^k$ for $h$ small enough
    since $\langle \widetilde{E}_i^h, \widetilde{E}_j^h \rangle = \delta_{ij} + h$
    by the approximation properties of the projectors.
    Applying a Gram-Schmidt
    orthonormalization procedure to both frames
    $\{\widetilde{t}_i^h\}_{i=1}^{d}$
    and
    $\{\widetilde{n}_i^h\}_{i=1}^{c}$ separately,
    constructs a discrete orthonormal frame
    $\{t_i^h\}_{i=1}^{d} \cup \{n_i^h\}_{i=1}^{c}$ on $K$.
    As the orthonormalization procedure involves only
    $C^{\infty}$ operations,
    $\| \widetilde{t}_i^h - t_i^h \|_{L^{\infty}(K)} \lesssim h$
    and
    $\|\widetilde{n}_i^h - n_i^h \|_{L^{\infty}(K)} \lesssim h$
    and thus the constructed discrete orthonormal normal and tangential frames
    satisfy the desired approximation property.
  \end{proof}
\end{lemma}
\begin{figure}[htb]
  \begin{center}
    \includegraphics[width=0.45\textwidth]{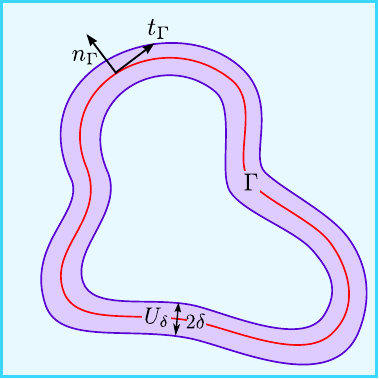}
    \hspace{0.03\textwidth}
    \includegraphics[width=0.45\textwidth]{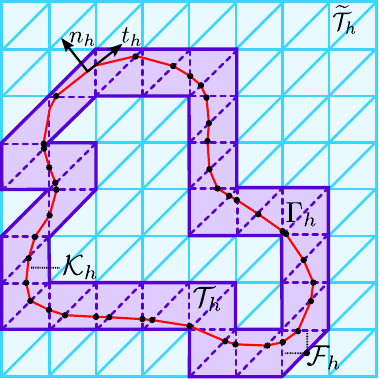}
  \end{center}
  \caption{Set-up of the continuous and discrete domains. (Left) Continuous surface $\Gamma$
  enclosed by a $\delta$ tubular neighborhood $U_{\delta}(\Gamma)$.
  (Right) Discrete manifold $\Gamma_h$ embedded into a background mesh
  $\widetilde{\mcT}_h$ from which the active (background) mesh $\mcT_h$ is extracted.
  }
  \label{fig:domain-set-up}
\end{figure}

\subsection{Stabilized Cut Finite Element Methods}
Let $\widetilde{\mcT}_{h}$ be a quasi-uniform mesh, with mesh parameter
$0<h\leq h_0$ and consisting of shape regular closed simplices,
of some open and bounded domain $\Omega \subset \RR^{k}$ containing
the embedding neighborhood $U_{\delta_0}(\Gamma$).
For the background mesh $\widetilde{\mcT}_{h}$ we define the \emph{active} (background)
$\mcT_h$ mesh and its set of \emph{interior} faces $\mcF_h$ by
\begin{align}
  \mcT_h &= \{ T \in \widetilde{\mcT}_{h} : T \cap \Gamma_h \neq \emptyset \},
  \label{eq:narrow-band-mesh}
  \\
  \mcF_h &= \{ F =  T^+ \cap T^-: T^+, T^- \in \mcT_h \},
  \label{eq:interior-faces}
\end{align}
and for the domain covered by $\mcT_h$ we introduce the notation
\begin{align}
  N_h &= \cup_{T \in \mcT_h} T.
  \label{eq:Nh-def}
\end{align}
Note that for any element
$T \in \mcT_h$ there is a neighbor $T'\in \mcT_h$ such that $T$ and
$T'$ share a face.
We assume that the partition $\mcK_h$ of $\Gamma_h$ is
compatible with the active mesh in the sense that $\foralls K \in \mcK_h$
it holds $ K \subset K \cap
T$ whenever $K \cap T \neq \emptyset$.
Such a compatible partition
$\mcK_h$ can always be generated starting from an initial
partition $\widetilde{\mcK}_h$ by subtriangulating non-empty
intersections $K \cap T$.
The various set of geometric entities are illustrated in
Figure~\ref{fig:domain-set-up}.

On the active mesh $\mcT_h$ we introduce
the discrete space of continuous piecewise linear polynomials,
\begin{equation}
  V_h = \{ v \in C(N_h) : v|_T \in  P_1(T) \; \foralls T \in \mcT_h \},
  \label{eq:Vh-def}
\end{equation}
and define the discrete counterpart of $H^1(\Gamma)/\RR$ to be the
function space consisting of those $v \in V_h$ with zero average
$\lambda_{\Gamma_h}(v) = \int_{\Gamma_h} v$,
\begin{align}
  V_{h,0} = \{ v \in V_h : \lambda_{\Gamma_h}(v) = 0 \}.
  \label{eq:Vh0-def}
\end{align}
Then the general form of the
stabilized cut finite element method for the
Laplace-Beltrami problem~\eqref{eq:LB}
is to seek
$u_h \in V_{h,0}$ such that
\begin{align}
  a_h(u_h,v) + \tau s_{h}(u_h, v) = l_h(v)
  \quad \forall v \in V_{h,0}.
  \label{eq:stabilized-cutfem-LB}
\end{align}
Here, $a_h$ and $l_h$ denote discrete counterparts of the continuous bilinear $a$
and linear form $l$, respectively, and are defined on $\Gamma_h$,
while $s_h$ represents a stabilization term,
which is weighted with a dimensionless
stabilization parameter $\tau > 0$. Both $a_h$ and $s_h$ are supposed to be symmetric.
The role of the stabilization
is to enhance the stability properties of the discrete bilinear form
$a_h$ in such a way that
geometrically robust optimal condition number and a priori
error estimates can be derived which are independent of the position
of $\Gamma_h$ relative to $\mcT_h$.
Additionally, to facilitate the abstract analysis of the
condition number and a priori error bounds for formulation~\eqref{eq:stabilized-cutfem-LB},
the discrete forms need to satisfy certain assumptions
which will be defined
in Section~\ref{sec:abstract-condition-number-estimates}
and \ref{sec:abstract-a-priori-analysis}.
Specific realizations will be discussed in Section~\ref{sec:cutfem-realizations}.

In the course of the forthcoming abstract numerical analysis, we will make
use of the stabilization (semi)-norm
\begin{align}
  \| v \|_{s_h}^2 = s_h(v,v),
  \label{eq:sh-norm}
\end{align}
as well as of the following energy norms defined for
for $v \in H^1(\Gamma) + V_h^l$ and $w \in H^1(\Gamma)^e + V_h$
\begin{align}
  \| v \|_{a}^2 = a(v,v), \quad
  \| w \|_{a_h}^2 = a_h(w,w),  \quad
  \| w \|_{A_h}^2 = A_h(w,w) = \|w\|_{a_h}^2 + \tau \|w\|_{s_h}^2,
  \label{eq:energy-norms}
\end{align}
where $A_h$ denotes the overall symmetric, discrete bilinear form
\begin{align}
  A_h(v,w) &= a_h(v,w) + \tau s_{h}(v,w) \quad \foralls v,w \in V_{h,0}.
  \label{eq:Ah-def}
\end{align}
Additionally, we assume that $\| \cdot \|_{A_h}$ defines a stronger norm
then $\| \cdot \|_{a}$ in the sense that
\begin{align}
  \| v \|_{a} \lesssim \| v^e \|_{A_h}, \qquad \| w^l \|_{a} \lesssim \| w \|_{A_h}.
  \label{eq:a-vs-Ah-norm}
\end{align}
Clearly, the abstract bilinear form~\eqref{eq:Ah-def}
is both coercive and continuous with respect to
$\|\cdot\|_{A_h}$
\begin{align}
  \|v\|_{A_h}^2 &\lesssim A_h(v,v),
  \qquad
  \\
  A_h(v,w) &\lesssim \| v \|_{A_h} \| w \|_{A_h}.
  \label{eq:Ah-boundedness}
\end{align}

\section{Realizations of the Abstract Cut Finite Element Method}
\label{sec:cutfem-realizations}
We now define and briefly discuss a number of concrete realizations
of the discrete bilinear form~$a_h$ and the stabilization form~$s_h$,
summarized in Table~\ref{tab:cutfem-form-realizations}.
For the discrete linear form~$l_h$, we simply assume that it is
given by
\begin{equation}
  l_h(v) = ( f^e, v )_{\mcK_h}.
  \label{eq:lh-def}
\end{equation}
For the discrete bilinear form $a_h$ we consider two variants,
one built upon the discrete tangential gradient
while the second variant replaces the discrete tangential gradient with the full gradient,
similar to the surface PDE methods considered in~\cite{DeckelnickElliottRanner2013,Reusken2014}.
\begin{table}[htb]
  \renewcommand{\arraystretch}{1.5}
  \begin{center}
    \begin{tabular}{ll}
      \toprule
      {\bf Discrete bilinear forms} & $a_h(u,v)$
      \\
      \midrule
      Tangential gradient & $a_h^1(v,w) = (\nablash v, \nablash w)_{\Gammah}$ \\
      Full gradient & $a_h^2(v,w) = (\nabla v, \nabla w)_{\Gammah}$ \\
      \toprule
      {\bf Stabilization forms} & $s_h(u,v)$
      \\
      \midrule
      Face-based & $s_h^1(v,w) = h^{1-c} (n_F \cdot \jump{\nabla v}, n_F \cdot \jump{\nabla w})_{\mcF_h}$
      \\
      Full gradient & $s_h^2(v,w) = h^{2-c}(\nabla v, \nabla w)_{\mcT_h}$
      \\
      Normal gradient & $s_h^3(v,w) = h^{\alpha-c}(  \Qsh^e \nabla v, \Qsh^e \nabla w)_{\mcT_h},
                      \qquad \alpha \in [0,2]$
      \\
      \bottomrule
    \end{tabular}
  \end{center}
  \caption{Realizations of the discrete bilinear form $a_h$ and stabilization form $s_h$.}
  \label{tab:cutfem-form-realizations}
\end{table}
Next, we recall that
stabilization operators $s_h$ using a face-based
gradient jump penalization and an artificial diffusion like, full gradient
stabilization were introduced
in \cite{BurmanHansboLarson2015}
and \cite{BurmanHansboLarsonEtAl2016c}
for various cut finite element discretizations
of the Laplace-Beltrami problem on surfaces.
Here, we generalized these stabilization operators to work
with cut finite element formulations on manifolds of arbitrary codimension.
Additionally, motivated by fact that the normal gradient  $\Qs \nabla v^e$ for any normal extension
of a function $v\in H^1(\Gamma)$ vanishes, we
propose a novel stabilization which penalizes the discrete normal gradient,
see Table~\ref{tab:cutfem-form-realizations}.

\begin{remark}
  \label{rem:face-based-vs-full-grad-stab}
  Compared to the face-based stabilization~$s_h^1$,
  the full gradient stabilization~$s_h^2$,
  has three main advantages: First, its implementation
  is extremely cheap and immediately available in many
  finite element codes. Second,
  the stencil of the discretization operator
  is not enlarged, as opposed to using a face-based penalty operator.
  Third, numerical studies for the surface case indicate that the
  accuracy and conditioning of a full gradient stabilized surface method
  is less sensitive to the choice of the stability parameter $\tau$ than
  for a face-based stabilized scheme, see~\cite{BurmanHansboLarsonEtAl2016c}.
\end{remark}
\begin{remark}
  While the full gradient stabilization has a number of advantages,
  its use is limited to low-order $P_1$ methods
  due to its inconsistency. The inconsistency of the normal gradient stabilization on the
  other hand is purely caused by geometric approximation errors of the normal field
  encoded in the mapping $\Qsh \circ \Ps$,
  see the proof of Lemma~(\ref{lem:normal-grad-est}).
  This has two consequences. First, it gives us more freedom for the choice of
  the $h$-scaling in $s_h^3$ and the possibility to tune the control of the normal gradient component
  of the computed discrete solution. Second, the stabilization is weakly consistent also
  for high-order $P_k$ methods $k \geqslant 1$ when the appropriate geometric
  approximation properties $\| \rho \|_{L^{\infty}(\Gammah)} \lesssim h^{k+1}$
  and $\|\Qs^e - \Qsh\|_{L^{\infty}(\Gammah)} \lesssim h^k$ are met.
\end{remark}

\section{Abstract Condition Number Estimates}
\label{sec:abstract-condition-number-estimates}
In this section, we formulate two abstract assumptions on the
stabilized bilinear form~\eqref{eq:Ah-def} which allow us to establish
optimal condition number bounds for the associated discrete system
which are independent of the position of the manifold $\Gamma$
relative to the background mesh~$\mcT_h$.
Following the approach in~\cite{ErnGuermond2006}, we require that for $v \in
V_{h,0}$,
the discrete bilinear form $A_h$ satisfies
\begin{itemize}
  \item {\bf a discrete Poincar\'e estimate}
    \begin{equation}
      h^{-c} \| v - \lambda_{\Gamma_h}(v) \|^2_{\mcT_h}
      \lesssim
      \| v \|_{A_h}^2,
      \label{eq:discrete-poincare-Ah-abstract}
    \end{equation}
  \item {\bf an inverse estimate} of the form
    \label{lem:inverse-estimate-Ah}
    \begin{align}
      \| v \|_{A_h}^2 \lesssim h^{-2-c} \| v \|_{\mcT_h}^2,
      \label{eq:inverse-estimate-Ah-abstract}
    \end{align}
\end{itemize}
with the hidden constants being independent of the manifold position
in the background mesh.
\begin{remark}
  Note that an immediate consequence of ~\eqref{eq:discrete-poincare-Ah-abstract}
  is that $\|\cdot\|_{A_h}$ indeed defines a norm on the normalized discrete space $V_{h,0}$.
\end{remark}
Next, we define the stiffness matrix $\mcA$ associated with the
bilinear form $A_h$ by the relation
\begin{align}
  ( \mcA V, W )_{\RR^N}  = A_h(v, w) \quad \foralls v, w \in
  V_h,
  \label{eq:stiffness-matrix}
\end{align}
where $V = \{V_i\}_{i=1}^N \in \RR^N$
is the expansion coefficient vector
for $v_h \in V_{h}$ with respect to the
standard piecewise linear basis functions
 $\{\phi_i\}_{i=1}^N$ associated with $\mcT_h$; that is,
$v = \sum_{i=1}^N V_i \phi_i$. Recall that for a quasi-uniform mesh $\mcT_h$,
the coefficient vector $V$ satisfies the well-known estimate
\begin{align}
  h^{k/2} \| V \|_{\RR^N}
  \lesssim \| v_h \|_{L^2(N_h)}
  \lesssim
  h^{k/2} \| V \|_{\RR^N}
  \label{eq:mass-matrix-scaling}
\end{align}
The fulfillment of the discrete Poincar\'e estimate~\eqref{eq:discrete-poincare-Ah-abstract}
ensures that the stiffness matrix $\mcA$ is a bijective linear mapping
$\mcA:\widehat{\RR}^N \to \ker(\mcA)^{\perp}$
where we set $\widehat{\RR}^N = \RR^N /\ker(\mcA)$
to factor out the one-dimensional kernel given by
$\ker{\mcA} = \spann\{(1,\ldots,1)^{T}\}$.
For the matrix $\mcA$, its operator norm and condition number are defined by
\begin{align}
  \| \mcA \|_{\RR^N}
  = \sup_{V \in \widehat{\RR}^N\setminus\bfzero}
  \dfrac{\| \mcA V \|_{\RR^N}}{\|V\|_{\RR^N}}
  \quad \text{and}
  \quad
  \kappa(\mcA) = \| \mcA \|_{\RR^N} \| \mcA^{-1} \|_{\RR^N},
  \label{eq:operator-norm-and-condition-number-def}
\end{align}
respectively. Then combining the mass matrix scaling~\eqref{eq:mass-matrix-scaling}
with the Poincar\'e estimate~\eqref{eq:discrete-poincare-Ah-abstract},
the inverse estimate~\eqref{eq:inverse-estimate-Ah-abstract} and the boundedness~\eqref{eq:Ah-boundedness}
of $A_h$,
the abstract approach in~\cite{ErnGuermond2006} allows
to establish the following bound for the condition number:
\begin{theorem}
  \label{thm:condition-number-estimate}
  The condition number of the stiffness matrix satisfies
  the estimate
  \begin{equation}
    \kappa( \mcA )\lesssim h^{-2},
  \end{equation}
  where the hidden constant depends only on the quasi-uniformness
  parameters.
\end{theorem}
\begin{proof} We need to bound $\| \mcA \|_{\RR^N}$ and $\| \mcA^{-1} \|_{\RR^N}$.
  First observe that for $w \in V_h$,
  \begin{equation}
    \| w \|_{A_h}
    \lesssim h^{-(2+c)/2} \| w \|_{N_h}
    \lesssim h^{(k-2-c)/2}\|W\|_{\RR^N}
    = h^{(d-2)/2}\|W\|_{\RR^N},
  \end{equation}
  where the inverse estimate~\eqref{eq:inverse-estimate-Ah-abstract}
  and equivalence~\eqref{eq:mass-matrix-scaling}
  were successively used.
  Consequently,
  \begin{align}
    \| \mcA V\|_{\RR^N} &= \sup_{W \in \RR^N }
    \frac{( \mcA V, W)_{\RR^N}}{\| W \|_{\RR^N}}
    = \sup_{w \in V_h }  \frac{A_h(v,w)}{\| w \|_{A_h}}
    \frac{\| w \|_{A_h}}{\| W \|_{\RR^N}}
    \lesssim h^{(d-2)/2} \| v \|_{A_h}
    \lesssim h^{d-2}\|V\|_{\RR^N},
  \end{align}
  and thus
  $
  \| \mcA \|_{\RR^N} \lesssim h^{d-2}
  $
  by the definition of the operator norm.
  To estimate $\| \mcA^{-1}\|_{\RR^N}$,
  start from \eqref{eq:mass-matrix-scaling} and combine the Poincar\'e
  inequality~\eqref{eq:discrete-poincare-Ah-abstract} with
  a Cauchy-Schwarz inequality to arrive at the following chain of
  estimates:
  \begin{align}
    \| V \|^2_{\RR^N}
    \lesssim h^{-k} \| v \|^2_{N_h}
    \lesssim h^{c-k} A_h(v,v)
    = h^{-d} (V, \mcA V)_{\RR^N}
    \lesssim h^{-d} \| V \|_{\RR^N} \| \mcA V \|_{\RR^N},
  \end{align}
  and hence $\| V \|_{\RR^N} \lesssim h^{-d}\| \mcA V\|_{\RR^N}$.
  Now setting $ V = \mcA^{-1} W$ we conclude that
  $
  \| \mcA^{-1}\|_{\RR^N} \lesssim h^{-d}
  $
  and combining estimates for $\| \mcA\|_{\RR^N}$ and $\| \mcA^{-1}\|_{\RR^N}$ the theorem follows.
\end{proof}

\section{Abstract A Priori Error Analysis }
\label{sec:abstract-a-priori-analysis}
This section is devoted to the abstract analysis of the a priori
error for the weak formulation~\eqref{eq:stabilized-cutfem-LB}.
First, we derive two abstract Strang-type lemmas which
show that the total energy and $L^2$ error
can be split into
interpolation, quadrature and consistency errors.
Then we present general assumptions the discrete forms $a_h$,
$s_h$ and $l_h$ must satisfy in order to ensure that the resulting cut finite element
method~(\ref{eq:stabilized-cutfem-LB})
defines a optimally convergent discretization scheme.

\subsection{Two Strang-type Lemma}
We start with a Strang-type lemma for the energy error.
\begin{lemma}
  \label{lem:strang-energy}
  Let $u$ the solution of \eqref{eq:LB}
  and $u_h$ the solution of \eqref{eq:stabilized-cutfem-LB}. Then
  \begin{align}
    \| u^e - u_h \|_{A_h}
    &\lesssim
    \inf_{v_h \in V_h}
    \| u^e - v_h \|_{A_h}
    + \sup_{v_h \in V_h}
    \dfrac{
      l_h(v_h) - A_h(u^e,v_h)
    }{
      \| v_h \|_{A_h}
    }
    \label{eq:strang-energy-1}
    \\
    &\lesssim
    \inf_{v_h \in V_h}
    \| u^e - v_h \|_{A_h}
    + \sup_{v_h \in V_h}
    \dfrac{
      l_h(v_h) - l(v_h^l)
    }{
      \| v_h \|_{A_h}
    }
    + \sup_{v_h \in V_h}
    \dfrac{
      a(u,v_h^l) - a_h(u^e,v_h)
    }{
      \| v_h \|_{A_h}
    }
    \nonumber
    \\
    &\phantom{\leqslant}
    + \sup_{v \in V_h}
    \dfrac{
      s_h(u^e,v_h)
    }{
      \| v_h \|_{A_h}
    }.
    \label{eq:strang-energy-2}
  \end{align}
\end{lemma}
\begin{proof}
  Thanks to triangle inequality
  $\| u^e - u_h \|_{A_h} \leqslant \| u^e - v_h \|_{A_h}
  + \| u_h - v_h \|_{A_h}$ with $v\in V_h$, it is enough to proceed with
  the ``discrete error'' $e_h = u_h - v_h$. Observe that
  \begin{align}
    \| e_h \|_{A_h}^2
    & = A_h(u_h - v_h,e_h)
    \\
    & =
    l_h(e_h) - A_h(u^e, e_h) + A_h(u^e - v_h,e_h)
    \\
    & \lesssim
    \biggl(
    \sup_{v \in V_h}
    \dfrac{
      l_h(v_h) - A_h(u^e,v_h)
    }{
      \| v_h \|_{A_h}
    }
    +
    \sup_{v_h \in V_h}
    \dfrac{
      A_h(u^e - v_h,v)
    }{
      \| v \|_{A_h}
    }
    \biggr)
    \| e_h \|_{A_h}.
    \label{eq:strang-energy-lemma-step-3}
  \end{align}
  Dividing by $\|e_h\|_{A_h}$
  and applying a Cauchy-Schwarz inequality
  to the second remaining term in~\eqref{eq:strang-energy-lemma-step-3}
  gives \eqref{eq:strang-energy-1}.
  To prove the second estimate~\eqref{eq:strang-energy-2}, simply observe that the identity
  \begin{align}
    l_h(v_h) - A_h(u^e, v_h)
    &=
    \bigl(
    l_h(v_h) -
    l(v_h^l)
    \bigr)
    +
    \bigl(a(u,v_h^l) - a_h(u^e, v_h)
    \bigr)- s_h(u^e, v_h).
  \end{align}
  follows directly from
  inserting $l(v_h^l) - a(u,v_h^l) = 0$  into ~\eqref{eq:strang-energy-lemma-step-3}.
\end{proof}
Next, we derive a corresponding representation for the $L^2$ error
using the standard Aubin-Nitsche duality argument.
\begin{lemma}
  \label{lem:strang-l2norm}
  With $u$ the solution of \eqref{eq:LB}
  and $u_h$ the solution of \eqref{eq:stabilized-cutfem-LB} it holds for any $\phi_h \in V_h$
  \begin{align}
    \| u^e - u_h \|_{\Gamma}
    &\lesssim
      \| u - u_h^l \|_{a} \sup_{\phi \in H^2(\Gamma)} \dfrac{\| \phi - \phi_h^l \|_{a}}{\|\phi\|_{2,\Gamma}}
      + \sup_{\phi \in H^2(\Gamma)}
      \dfrac{
        l(\phi_h) - a(u_h^l,\phi_h^l)
      }{
        \| \phi \|_{2,\Gamma}
      }
    + \| \lambda_{\Gamma}(u_h^l) \|_{\Gamma}
      \label{eq:strang-l2norm-1}
    \\
    &\lesssim
      \| u - u_h^l \|_{a} \sup_{\phi \in H^2(\Gamma)} \dfrac{\| \phi - \phi_h^l \|_{a}}{\|\phi\|_{2,\Gamma}}
      + \sup_{\phi \in H^2(\Gamma)}
      \dfrac{
        l(\phi_h^l)
      -
      l_h(\phi_h)
    }{
      \| \phi \|_{2,\Gamma}
    }
      \nonumber
    \\
    &\phantom{\leqslant}
      + \sup_{\phi \in H^2(\Gamma)}
      \dfrac{
        a_h(u_h,\phi_h) -
      a(u_h^l,\phi_h^l)
    }{
      \| \phi \|_{2,\Gamma}
    }
      + \sup_{\phi \in H^2(\Gamma)}
      \dfrac{
        s_h(u_h,\phi_h)
      }{
        \| \phi \|_{2,\Gamma}
      }
    + \| \lambda_{\Gamma}(u_h^l) \|_{\Gamma}.
      \label{eq:strang-l2norm-2}
  \end{align}
\end{lemma}
\begin{proof}
  First,
  we decompose the error $e = u - u_h^l$
  into a normalized part $\tilde{e}$ satisfying
  $
  \lambda_{\Gamma}(\tilde{e}) = 0
  $
  and a constant part,
  \begin{align}
    e = u - u_h^l = \underbrace{u - (u_h^l - \lambda_\Gamma(u_h^l))}_{\tilde{e}}
    - \lambda_\Gamma(u_h^l).
  \end{align}
  By the triangle inequality $\|e\|_{\Gamma} \leqslant \|
  \tilde{e} \|_{\Gamma} + \| \lambda_{\Gamma}(u_h^l) \|_{\Gamma}$, it
  suffices to proceed with $\| \tilde{e} \|_{\Gamma}$.
  Now let $\psi \in L_0^2(\Gamma)$ and take $\phi \in
  H^2(\Gamma)$ satisfying $-\Delta_{\Gamma} \phi = \psi$
  and the elliptic regularity estimate
  $\|\phi \|_{2,\Gamma} \lesssim \| \psi \|_{\Gamma}$, see~(\ref{eq:ellreg}).
  Then the normalized error $\tilde{e}$ can be represented as
  $(\tilde{e}, \psi)_{\Gamma}
    = a(\tilde{e}, \phi) = a(e, \phi)
  $
  and
  adding and subtracting any lifted finite element function $\phi_h^l$
  gives
  \begin{align}
    \| \tilde{e} \|_{\Gamma}
    &= \sup_{\psi \in L_0^2(\Gamma)}
    \dfrac{(\tilde{e}, \psi)}{\|\psi\|_{\Gamma}}
    \\
    &\lesssim \sup_{\phi \in H^2(\Gamma)}
    \dfrac{a(e, \phi)}{\|\phi\|_{2,\Gamma}}
    \\
    &\lesssim \sup_{\phi \in H^2(\Gamma)}
    \dfrac{a(e, \phi - \phi_h^l)}{\|\phi\|_{2,\Gamma}}
    + \sup_{\phi \in H^2(\Gamma)}
    \dfrac{a(e, \phi_h^l)}{\|\phi\|_{2,\Gamma}}
    \\
    &\lesssim
    \| u - u_h^l \|_a
    \sup_{\phi \in H^2(\Gamma)} \dfrac{\| \phi - \phi_h^l\|_a}{\|\phi\|_{2,\Gamma}}
    + \sup_{\phi \in H^2(\Gamma)}
    \dfrac{l(\phi_h^l) - a(u_h^l, \phi_h^l)}{\|\phi\|_{2,\Gamma}},
    \label{eq:strang-l2norm-lemma-step-3}
  \end{align}
  which proves \eqref{eq:strang-l2norm-1}. Similar as in the proof of the
  previous Strang Lemma, the second estimate~\eqref{eq:strang-l2norm-2}
  follows then from inserting
  $a_h(u_h, \phi_h) + s_h(u_h, \phi_h) - l_h(\phi_h) = 0$
  into the second term of ~\eqref{eq:strang-l2norm-lemma-step-3}.
\end{proof}

\subsection{A Priori Error Estimates}
\label{ssec:a-priori-est}
Motivated by the abstract Strang-type lemma for the energy and
$L^2$ norm error, we now assume that the following estimates hold in
order to derive optimal bounds for the a priori error of the abstract cut finite
element formulation~(\ref{eq:stabilized-cutfem-LB}):
\begin{itemize}
  \item {\bf Interpolation estimates.} There exists
    an interpolation operator $\pi_h: H^2(\Gamma) \to V_h$ such that
    for $v \in H^2(\Gamma)$ it holds
    \begin{align}
      \| v^e - \pi_h v^e \|_{\Gamma_h} +
      h \| v^e - \pi_h v^e \|_{A_h} &\lesssim h^2 \|v \|_{2,\Gamma}.
      \label{eq:interpolation-est-req-Ah}
    \end{align}
  \item {\bf Quadrature estimates.}
    To prove optimal energy error estimates, we assume
    that for $v \in V_{h}$ and a finite element approximation $u_h \in V_h$
    of $u \in H^2(\Gamma)$, it holds
    \begin{align}
      | l_h(v) - l(v^l) | &\lesssim h \| f \|_{\Gamma} \|v \|_{A_h},
                          \label{eq:quadrature-l-est-primal-req}
      \\
      | a(u,v^l) - a_h(u^e,v) | &\lesssim h \| u \|_{2,\Gamma} \| v \|_{A_h}.
                                \label{eq:quadrature-a-est-primal-req}
    \end{align}
    Moreover, in order to obtain an optimal bound for the $L^2$ error
    using the standard Nitsche-Aubin duality trick, we
    require that the improve estimates
    \begin{align}
      \label{eq:quadrature-l-est-dual-req}
      | l_h(\phi_h) - l(\phi_h^l) | &\lesssim h^2 \| f \|_{\Gamma} \|\phi\|_{2,\Gamma},
      \\
      \label{eq:quadrature-a-est-dual-req}
      | a(u_h^l,\phi_h^l) - a_h(u_h,\phi_h) | &\lesssim h^2 \| u \|_{2,\Gamma} \| \phi \|_{2,\Gamma}
    \end{align}
    hold, whenever $\phi_h$ is a suitable finite element approximation of
    $\phi \in H^2(\Gamma)$.
  \item {\bf Consistency error estimate.} Finally, we require that the
    stabilization term $s_h$ is weakly consistent in the sense that
    for $\foralls v \in H^2(\Gamma)$
    \begin{align}
      \|v^e\|_{s_h} \lesssim h \| v \|_{2,\Gamma}.
      \label{eq:consist-err-est-req}
    \end{align}
\end{itemize}
If these assumptions are met, it is easy to prove the following
theorem.
\begin{theorem}
  Let $u \in H^2(\Gamma)$ be the solution to continuous problem~(\ref{eq:LB})
  and $u_h$ be the solution to the discrete problem~(\ref{eq:stabilized-cutfem-LB}).
  Then the following a priori error estimates hold
  \label{thm:aprioriest}
  \begin{align}
    \label{eq:energyest}
    \| u^e - u_h \|_{A_h} &\lesssim h \| f \|_{\Gamma},
    \\ \label{eq:ltwoest}
    \| u - u_h^l \|_{\Gammah} &\lesssim h^2 \| f \|_{\Gamma}.
  \end{align}
\end{theorem}
\begin{proof}
  The proof of the energy estimate~(\ref{eq:energyest}) follows directly
  from the Strang lemma~(\ref{eq:strang-energy-2}) and
  assumptions~(\ref{eq:interpolation-est-req-Ah}),
  (\ref{eq:quadrature-l-est-primal-req})--(\ref{eq:quadrature-a-est-primal-req})
  and~(\ref{eq:consist-err-est-req}),
  only noting that $s_h(u_h, v) \leqslant \| u_h \|_{s_h} \| v \|_{A_h}$
  thanks to the symmetry of $s_h$ and the definition of
  $\|\cdot\|_{A_h}$.

  To prove the $L^2$ error estimate~(\ref{eq:ltwoest}),
  it only remains to have a closer look at the first and two last terms
  in Strang lemma~(\ref{eq:strang-l2norm-2}).
  Set $\phi_h  = \pi_h \phi$ and
  observe that by assumption~(\ref{eq:a-vs-Ah-norm}),
  $\| u - u_h^l \|_a \lesssim \| u_h^e - u_h \|_{A_h}$
  and
  $\| \phi - \phi_h^l \|_a \lesssim \| \phi^e - \phi_h \|_{A_h}$
  and consequently,
  $
  \| u - u_h^l \|_a \| \phi - \phi_h^l \|_a
  \lesssim
  h^2 \| f \|_{2,\Gamma} \| \phi\|_{2,\Gamma}
  $
  in the first term in~(\ref{eq:strang-l2norm-2}).
  Next, the estimates for the energy, interpolation and consistency error
  give in combination with a Cauchy-Schwarz inequality
  the following bound
  \begin{align}
    s_h(u_h,\phi_h)
    &= s_h(u_h - u^e, \phi_h) + s_h(u^e, \phi_h)
    \\
    &=
    s_h(u_h - u^e, \phi_h - \phi^e)
    +s_h(u_h - u^e, \phi^e)
    +s_h(u^e, \phi^e)
    +s_h(u^e, \phi_h - \phi^e)
    \\
    &\lesssim
    h \| f \|_{2,\Gamma} h \|\phi \|_{2,\Gamma}.
  \end{align}
  Finally, to estimate the deviation of the lifted function $u_h^l$ from
  the $0$-average encoded in $\|\lambda_{\Gamma}(u_h^l) \|_{\Gamma}$,
  simply insert $\lambda_{\Gammah}(u_h) = 0$
  and unwind the definition of the average operators
  $\lambda_{\Gammah}(\cdot)$ and $\lambda_{\Gamma}(\cdot)$
  to see that
  \begin{align}
    \|\lambda_{\Gamma}(u_h^l) \|_{\Gamma}
    =
    |\Gamma |^{\onehalf}
        \left|
    \dfrac{1}{|\Gamma|}
    \int_{\Gamma} u_h^l  \dGamma
    -
    \dfrac{1}{|\Gamma_h|}
    \int_{\Gamma_h} u_h  \dGammah
    \right|
    \lesssim
    \dfrac{|\Gamma|^{\onehalf}}{|\Gamma_h|}
    \int_{\Gamma_h} |1-c| |u_h|  \dGammah,
        \label{eq:error-of-average-est}
  \end{align}
  with $c = |\Gamma_h||\Gamma|^{-1} |B|$.
  Anticipating the geometric estimates to be established in
  Section~\ref{ssec:domain-change}, we observe that $\| 1 - c
  \|_{L^\infty(\Gamma)} \lesssim h^2$
  thanks to~\eqref{eq:detBbound}.
  Consequently, after
  successively applying a Cauchy-Schwarz inequality, the inverse
  estimate~\eqref{eq:inverse-estimate-cut-v-on-K},
  the discrete Poincar\'e estimate~\eqref{eq:discrete-poincare-Ah-abstract},
  and finally, the stability bound $\|u_h\|_{A_h} \lesssim \|
  f\|_{\Gamma}$,
  we arrive at the desired estimate,
  \begin{align}
    \|\lambda_{\Gamma}(u_h^l) \|_{\Gamma}
    \lesssim
    \dfrac{|\Gamma|}{|\Gammah|^{\onehalf}}h^2 \| u_h\|_{L^1({\Gamma})}
    \lesssim
    \dfrac{|\Gamma|}{|\Gammah|}h^2 \| u_h\|_\Gamma
    \lesssim
    h^2 h^{-\onehalf} \| u_h\|_{\mcT_h}
    \lesssim
    h^2 \| u_h\|_{A_h}
    \lesssim
    h^2 \|f \|_{\Gamma}.
  \end{align}
\end{proof}

\section{Geometric Estimates and Properties}
\label{sec:geometric-estimates}
The aim of this section is to develop and collect
those geometric properties, identities and estimates
which will be needed in the forthcoming verification
of the abstract key
assumptions
~(\ref{eq:discrete-poincare-Ah-abstract})--(\ref{eq:inverse-estimate-Ah-abstract}),
(\ref{eq:quadrature-l-est-primal-req})--(\ref{eq:quadrature-a-est-dual-req})
and~(\ref{eq:consist-err-est-req}),
as well as in
construction of a suitable
interpolation operator satisfying ~(\ref{eq:interpolation-est-req-Ah}).
The main challenge is to generalize the well-known geometric estimates
given in
\cite{Dziuk1988,OlshanskiiReuskenGrande2009,DemlowDziuk2007,DziukElliott2013,BurmanHansboLarson2015,BurmanHansboLarsonEtAl2015b}
to the case of embedded manifolds $\Gamma$ of arbitrary codimensions,
where an explicit representation of the closest point projection is
not immediately available.
Consequently, estimates of
related expression such as derivatives of lifted and extend functions
must be established by an alternative route.
The route taken here is
based on introducing a special tube coordinate system
which is particularly well-adapted to perform computations in the tubular
neighborhood of $\Gamma$,  see \cite{Weyl1939,Gray2012} for a detailed
presentation and more advanced theoretical applications. 
Tubes coordinates allow us to derive a semi-explicit
representation of the derivative of the closest point projection
as well as useful local trace and Poincar\'e-type inequalities
for parts of the tubular neighborhood.
After providing a short and general proof for estimating
the change of the Riemannian measure when passing between
discrete and continuous manifolds, we conclude this
section with formulating and proving certain 
fat intersection properties of the discrete manifold~$\Gamma_h$.

\subsection{Tube Coordinates}
\label{ssec:tube-coordinates}
By the compactness of $\Gamma$ and a partition of unity argument it is
enough to consider a local parametrization
$\alpha : V \subset \RR^d \to \alpha(V) \subset \Gamma$
for which a smooth orthonormal normal frame $\{n_i\}_{i=1}^{c}$
exists on $\alpha(V)$.
Set $B_{\delta}^c(0) = \{ s \in \RR^{c} : \| s \| < \delta \}$
and define the mapping
\begin{align}
  \Phi: V \times B^c_{\delta}(0)
  \ni (y,s) \to  \alpha(y) + \sum_{i=1}^{c} s_i n_i(\alpha(y)) \in U_{\delta}(\Gamma).
  \label{eq:tube-coordinates}
\end{align}
We now show that $\Phi$ actually defines a diffeomorphism
by examining its derivative $D\Phi$ more closely.
First, observe that $\|s\| = \rho(x)$ for $x = \Phi(y,s)$ and thus we
simply write $\| s \| = \rho$. Computing $D\Phi$ gives
\begin{align}
  D \Phi
  &=
  \begin{pmatrix}
    \dfrac{\partial\alpha_1}{\partial y_1} + \sum_{i=1}^c s_i \dfrac{\partial }{\partial y_1}
    (n_i \circ \alpha)
    & \ldots &
    \dfrac{\partial\alpha_1}{\partial y_d} + \sum_{i=1}^c s_i \dfrac{\partial}{\partial y_d}
    (n_i \circ \alpha)
    & n_{1}^1& \cdots
    & n_{c}^1&
    \\
    \hdotsfor{6}
    \\
    \dfrac{\partial\alpha_k}{\partial y_1} + \sum_{i=1}^c s_i \dfrac{\partial }{\partial y_1}
    (n_i \circ \alpha)
    & \ldots &
    \dfrac{\partial\alpha_k}{\partial y_d} + \sum_{i=1}^c s_i \dfrac{\partial}{\partial y_d}
    (n_i \circ \alpha)
    & n_{1}^k& \cdots
    & n_{c}^k&
  \end{pmatrix}
  \\
  &=
  \underbrace{
    \left(
  D_y \alpha, n_{1}, \ldots, n_c
  \right)
}_{A}
  + \sum_{i=1}^c s_i
  \bigl(
  D_y (n_i \circ \alpha), \underbrace{0,\ldots,0}_{c\text{ zeros}}
  \bigr)
  \label{eq:dPhi}
\end{align}
Clearly, the matrices $D_y (n_i \circ \alpha)$ are bounded on $V$ and
thus, since the columns of $D_y \alpha$ span the tangential space $T_p \Gamma$,
the matrix
$D\Phi$ admits a decomposition
\begin{align}
  D\Phi = A - \rho S,
\end{align}
with
$A$ being invertible and
$\| S \|_{L^{\infty}(V)} \lesssim 1$.
Consequently,
for $\delta_0 < \| S A^{-1}\|_{L^{\infty}(V)}$,
$\Phi$ is a local and thus also a global diffeomorphism
by the bijectivity of the mapping $\Psi$ defined in \eqref{eq:def-Psi}.
We recall that given a local parametrization $\alpha$ for $\Gamma$,
the Riemannian measure $\dGamma$
is given by $\dGamma(y) = \sqrt{g^{\alpha}(y)}$ where
$g^{\alpha}(y) = D_y^{T}\alpha(y) D_y\alpha(y)$
is the metric tensor given in local coordinates $y = (y^1,\ldots,y^d)$.
Rewriting $\det D\Phi = \det A \det (I - \rho S A^{-1})$
and observing that
$\det(I - \rho SA^{-1}) \sim 1$
in the $\|\cdot \|_{L^{\infty}(V)}$ norm
for $\delta < \delta_0$ small enough,
we see that
\begin{align}
  \sqrt{g^{\Phi}(y,s)} =
  |\det D\Phi(y,s)|
  \sim |\det A| = (\det A^T A)^{\onehalf} = (\det (D^T\alpha(y)
    D\alpha(y))^{\onehalf}
  = \sqrt{g^{\alpha}(y)}.
  \label{eq:riemann-measure-equiv-I}
\end{align}

We conclude this section by introducing a ``sliced'' variant of the
$\delta$-tubular neighborhood.
For any $d$-dimensional measurable set $W \subset \Gamma \subset \RR^k$
and $\delta < \sqrt{k} \delta_0$, we introduce the
$(d+i)$-dimensional ``partial cubular'' neighborhood
\begin{align}
  Q^i_{\delta}(W)
  = \{ \RR^k \ni  p + \sum_{j=1}^i s_j n_j(p)  : p \in W \subset \Gamma \wedge \| s
    \|_{\infty} < \delta \}.
  \label{eq:cubular-neighborhood}
\end{align}
Note that we here chose the
maximum norm instead of the Euclidean norm. Clearly,
$
Q^c_{\sqrt{k} \delta}(\Gamma) \subset U_{\delta}(\Gamma) \subset Q_{\delta}^c(\Gamma)
$.
Similar as before, we can define a parametrization $\Phi^i$ defined by
\begin{align}
  \Phi^i: V \times Q^i_{\delta}(0)
  \ni (y,s) \to  \alpha(y) + \sum_{j=1}^{i} s_j n_j(\alpha(y)) \in U^i_{\delta}(W),
  \label{eq:cubular-coordinates}
\end{align}
where $Q_{\delta}^i(0) = \{ s \in \RR^{i} : \| s \|_{\infty} < \delta \}$ is the hypercube of
dimension $i$ and length $2\delta$.
Following the previous line of thought, we observe that
for $i,j \in \{0,\ldots,c\}$ and $\delta < \delta_0$ small enough
\begin{align}
  \sqrt{g^{\Phi^i}} \sim \sqrt{g^{\alpha}} \sim \sqrt{g^{\Phi^j}}.
  \label{eq:riemann-measure-equiv-II}
\end{align}

Partial cubular neighborhoods will be instrumental in deriving 
Poincar\'e-type inequalities and
interpolation estimates
in Section~\ref{ssec:analysis-normal-grad-stab}
and Section~\ref{sec:interpolation-properties}, respectively.
There, a common theme is to pass from $\Gamma$ to
its full tubular neighborhood $U_{\delta}(\Gamma)$ and vice versa 
by successively ascending from or descending to the
$i$-th cubular neighborhoods $Q_{\delta}^i$ 
defined in (\ref{eq:cubular-neighborhood})
employing the following scaled trace and Poincar\'e inequalities.
\begin{lemma}
  \label{lem:cubular-trace-estimate}
  Assume that $W$ is an open coordinate neighborhood in $\Gamma$ with
  a parametrization $\alpha: V \to W \subset \Gamma$, $V\subset \RR^d$, and a
  continuously defined normal bundle.
  Let $ w \in H^1(U_{\delta}^{i}(V))$ and $i \in \{1,\ldots,c-1\}$. Then
  for $\delta \leqslant \delta_0$ small enough, 
  the scaled {\bf trace inequality}
  \begin{align}
    \| w \|_{Q_{\delta}^{i-1}(W)}^2
    \lesssim \delta^{-1}
    \| w \|_{Q_{\delta}^{i}(W)}^2 +
    \delta \| \nabla w \|_{Q_{\delta}^{i}(W)}^2,
    \label{eq:cubular-trace-estimate}
    \\
    \intertext{holds as well as the scaled {\bf Poincar\'e inequality}}
    \| w \|_{Q_{\delta}^{i}(W)}^2 +
    \lesssim 
    \delta \| w \|_{Q_{\delta}^{i-1}(W)}^2
    +
    \delta^2 \| \nabla w \|_{Q_{\delta}^{i}(W)}^2.
    \label{eq:cubular-poincare-estimate}
  \end{align}
\end{lemma}
\begin{proof}
  We start with the proof for trace inequality~\eqref{eq:cubular-trace-estimate}.
  By an approximation argument, it is enough to assume that $w \in C^1(Q_{\delta}^i(W))$.
  Rewrite the integral
  $
  \| w \|_{Q_{\delta}^{i-1}(W)}^2
  $
  using the tube coordinates~(\ref{eq:cubular-coordinates}) and the measure
  equivalence~(\ref{eq:riemann-measure-equiv-II}) to see that
  \begin{align}
    \| w \|_{Q_{\delta}^i(W)}^2
    &=
    \int_{V}
    \left(
    \int_{Q_{\delta}^i}
    | w(y,s) |^2  \sqrt{g^{\Phi_i}(y,s)}
    \ds
    \right)
    \dy
    \sim
    \int_{V}
    \left(
    \int_{Q_{\delta}^i}
    | w(y,s) |^2
    \ds
    \right)
    \sqrt{g^{\alpha}(y)}
    \dy,
    \label{eq:integral-equiv}
  \end{align}
  Fixing $y$, the fundamental theorem of calculus allows us to
  rewrite the integrand $v(s) = w(y,s)$ as
  \begin{align}
    v(s_1,\ldots,s_i)
    =
    v(s_1,\ldots,s_{i-1},0)
    +
    \int_{0}^{s_i}
    \partial_{s_i} v(s_1,\ldots,s_{i-1},s) \ds_i,
    \label{eq:fund-calc-repres}
  \end{align}
  and consequently, after rearranging terms and a Cauchy-Schwarz inequality,
  \begin{align}
    |v(s_1,\ldots,s_{i-1},0)|^2
    &\lesssim
    |v(s)|^2
    +
    \left(\int_{-\delta}^{\delta}
    |\partial_{s_i} v(s_1,\ldots,s_{i-1},s_i)| \ds_i
    \right)^2
    \\
    &\lesssim
    |v(s)|^2
    +
    \delta \int_{-\delta}^{\delta}
    |\partial_{s_i} v(s_1,\ldots,s_{i-1},s_i)|^2 \ds_i.
  \end{align}
  Integrating the previous inequality over $Q_{\delta}^i(0)$ gives
  \begin{align}
    \delta \int_{Q_{\delta}^{i-1}}
    |v(s_1,\ldots,s_{i-1}, 0)|^2 ds
    &\lesssim
    \int_{Q_{\delta}^i}
    | v(s) |^2  ds
    + \delta^2
    \int_{Q_{\delta}^i}
    |\partial_{s_i} v(s)|^2
    \ds
  \end{align}
  and a subsequent integration over $V$ together with
equivalence~\eqref{eq:integral-equiv} finally leads us to
\begin{align}
  \delta \| w \|_{Q_{\delta}^{i-1}(W)}^2
  &\lesssim
  \| w \|_{Q_{\delta}^i(W)}^2
  +
  \delta^2 \| \nabla w \|_{Q_{\delta}^i(W)}^2.
\end{align}
Finally, observe that
starting from the representation~\eqref{eq:fund-calc-repres}
and rearranging terms properly,
the Poincar\'e inequality~\eqref{eq:cubular-poincare-estimate} can 
be proven in the exact same manner.
\end{proof}

\subsection{Gradient of Lifted and Extended Functions}
Next, using tube coordinates, we derive a semi-explicit representation
of the derivative of the closest point projection.
\begin{lemma}
  \label{lem:dp-representation}
  Whenever $\delta \leqslant \delta_0$ for $\delta_0$ small enough,
  the derivative $Dp$ of the closest point projection $p: U_{\delta}(\Gamma)
  \to \Gamma$
  admits a representation of the form
  \begin{align}
    Dp = \Ps(I - \rho \mcH),
    \label{eq:dp-representation}
  \end{align}
  with a matrix-valued function $\mcH$ satisfying
  $\| \mcH \|_{L^{\infty}(U_{\delta}(\Gamma))} \lesssim 1 $ and $\mcH \Ps = \mcH$.
\end{lemma}
\begin{proof}
  Using local tube coordinates~\eqref{eq:tube-coordinates}, we observe
  that the closest point projection $p$ and its derivative are
  given by
  \begin{align}
    p &= \alpha \circ \Pi_d \circ \Phi^{-1},
    \label{eq:p-representation-I}
    \\
    Dp &= D\alpha \circ \Pi_d \circ (D\Phi)^{-1},
    \label{eq:dp-representation-I}
  \end{align}
  where $\Pi_d: \RR^k \to \RR^d$ is the projection on the first $d$ components.
  Starting from the decomposition
  $D\Phi = A - \rho S$ derived in the previous section,
  insert a power series representation for the inverse matrix
  $(A - \rho S)^{-1} = A^{-1}(I - \rho S A^{-1})^{-1}$
  into~\eqref{eq:dp-representation-I}
  to conclude that
  \begin{align}
    Dp &= D\alpha \circ \Pi_d \circ A^{-1}( I + \rho\sum_{l=1}^{\infty}\rho^{l-1} (SA^{-1})^l),
    \\
    &= D\alpha \circ \Pi_d \circ A^{-1}( I - \rho\mcH),
    \label{eq:dp-representation-II}
  \end{align}
  with the absolutely and uniformly convergent power series
  $\mcH = -\sum_{l=1}^{\infty}\rho^{l-1} (SA^{-1})^l$.
  To arrive at representation~(\ref{eq:dp-representation}),
  it remains to show that  $ D\alpha \circ \Pi_d \circ A^{-1} = \Ps$.
  Setting $w = A^{-1} v$, a simple computation yields
  \begin{align}
    (D\alpha \circ \Pi_d \circ A^{-1}) v
    = D\alpha ( \Pi_d w )
    = \Ps ((D \alpha, n_{1}, \ldots, n_{c}) w)
    = \Ps A w  = \Ps v
  \end{align}
  for any $v \in \RR^k$
  since $\Ps(\partial_i \alpha) = \partial_i \alpha$ and $\Ps(n_i) = 0$.
  Finally, we demonstrate that $\mcH \Ps = \mcH$, or equivalently,
  that $\mcH n_i = 0$ for $i=1,\ldots,c$. But referring back to ~\eqref{eq:dPhi},
  we see that $SA^{-1}$
  consists of a (weighted) sum of matrices of the form
  \begin{align}
    (D (n_i \circ \alpha),0,\ldots,0) A^{-1}&=
    D (n_i \circ \alpha) \circ \Pi_d \circ A^{-1}
    \\
    &= D_x n_i(\alpha(\cdot)) \circ D\alpha \circ \Pi_d \circ A^{-1}
    \\
    &= D_x n_i(\alpha(\cdot)) \circ  \Ps
  \end{align}
  and thus $\mcH \Ps= \mcH$ which concludes the proof.
\end{proof}
Based on the previous lemma, estimates for the gradient of lifted and extended functions
can be established. Starting from $Dv^e = Dv \circ Dp$ and the definition
of the gradient, we conclude that $\foralls  a \in \RR^k$
\begin{align}
  \scp{\nabla v^e, a}
  = (Dv \circ Dp)a
  = Dv\Ps(I - \rho \mcH) a
  = \mean{(I - \rho \mcH^T)\Ps \nabla v, a}
\end{align}
and thus
\begin{align}
  \nabla v^e
  &= (I - \rho \mcH^T)\Ps \nabla v
   = (I - \rho \mcH^T) \nablas v
   = \Ps(I - \rho \mcH^T) \nablas v,
  \label{eq:ve-full-gradient}
  \\
  \nablash v^e &= \Psh(I - \rho \mcH^T)\Ps \nabla v = B^{T} \nablas v,
  \label{eq:ve-tangential-gradient}
\end{align}
where the invertible linear mapping
\begin{align}
  B = P_{\Gamma}(I - \rho H) P_{\Gamma_h}: T_x(\Gammah) \to T_{p(x)}(\Gamma)
  \label{eq:B-def}
\end{align}
maps the tangential space of $\Gamma_h$ at $x$ to the tangential space of $\Gamma$ at
$p(x)$. Setting $v = w^l$ and using the identity $(w^l)^e = w$, we immediately get that
\begin{align}
  \nablas w^l = B^{-T} \nablash w
\end{align}
for any elementwise differentiable function $w$ on $\Gamma_h$ lifted to $\Gamma$.
Similar to the standard hypersurface case $d = k-1$
in \cite{Dziuk1988,DziukElliott2013},
the following bounds for the linear
operator $B$ can be derived.
\begin{lemma} It holds
  \label{lem:BBTbound}
  \begin{equation}
    \| B \|_{L^\infty(\Gamma_h)} \lesssim 1,
    \quad \| B^{-1} \|_{L^\infty(\Gamma)} \lesssim 1,
    \quad
    \| P_\Gamma - B B^T \|_{L^\infty(\Gamma)} \lesssim h^2.
    \label{eq:BBTbound}
  \end{equation}
\end{lemma}
\begin{proof}
  Thanks to the representation~\eqref{eq:dp-representation}, the proof
  follows standard arguments, see~\citet{DziukElliott2013}, and is only
  sketched here for completeness.
  The first two bounds follow directly from Lemma~\ref{lem:dp-representation}.
  Using the assumption $\|\rho\|_{L^{\infty}(\Gamma_h)} \lesssim h^2$,
  it follows that
  $
  \Ps - B B^T = \Ps - \Ps \Psh \Ps + O(h^2).
  $
  An easy calculation now shows that
  $\Ps - \Ps \Psh \Ps = \Ps (\Ps - \Psh)^2 \Ps$
  from which the desired bound follows by observing that
  \begin{align}
    \Ps - \Psh =
    \sum_{i=1}^d
    \bigl(
    (t_i - t_i^h)  \otimes t_i
    +
    t_i^h \otimes (t_i - t_i^h)
    \bigr)
  \end{align}
  and thus
  $
  \| (\Ps - \Psh)^2 \|_{L^{\infty}(\Gamma_h)} \lesssim
  \sum_{i = 1}^d\| t_i - t_i^h \|_{L^{\infty}(\Gamma_h)}^2 \lesssim h^2
  $.
\end{proof}
To estimate the quadrature error for the full gradient form $a_h^2$,
we will need to quantify the error introduced by using the
full gradient $\nabla$ in~\eqref{eq:Ah-def} instead of $\nablash$.
To do so, we decompose the full gradient as
$\nabla = \nablash + \Qsh \nabla$.
An estimate for the normal component is provided by
\begin{lemma}
  \label{lem:normal-grad-est}
  For $v \in H^1(\Gamma)$ it holds
  \begin{align}
    \| \Qsh \nabla v^e \|_{\Gamma}
    \lesssim h
    \| \nablas v \|_{\Gamma}.
    \label{eq:normal-grad-est}
  \end{align}
\end{lemma}
\begin{proof}
  Since $\nabla v^e$ = $ \Ps(I - \rho \mcH^T)\nablas v$
  according to identity~\eqref{eq:ve-full-gradient}, it is enough to
  prove that
  \begin{align}
    \| \Qsh \Ps \|_{L^{\infty}(\Gamma)} \lesssim h.
    \label{eq:normal-tangential-est}
  \end{align}
  But using representation~\eqref{eq:gamma-projectors-complimentary},
  a simple computation shows that
  \begin{align}
    \| \Qsh \Ps \|_{L^{\infty}(\Gamma)}
    &\lesssim \sum_{i=1}^c
    \bigl(
    n_i^h \otimes n_i^h - \scp{n_i^h,n_i} n_i^h \otimes n_i \|_{L^{\infty}(\Gamma)}
    \bigr)
    \\
    &\lesssim \sum_{i=1}^c
    \bigl(
    \|(1 - \scp{n_i^h,n_i}) n_i^h \otimes n_i^h \|_{L^{\infty}(\Gamma)} + \| \scp{n_i^h,n_i} n_i^h
    \otimes (n_i - n_i^h) \|_{L^{\infty}(\Gamma)}
    \bigr)
    \\
    &\lesssim h^2 + h,
  \end{align}
  where we used the identity
  $
  1 - \scp{n_i^h,n_i} =
  \onehalf \scp{n_i^h - n_i, n_i^h - n_i}
  $
  and approximation assumption~(\ref{eq:normal-frame-est}).
\end{proof}

\subsection{Change of Domain Integration}
\label{ssec:domain-change}
Next, we derive estimates for the change of the Riemannian measure
when integrals are lifted from the discrete surface
to the continuous surface and vice versa.
For a subset $\omega\subset \Gammah$,
we have the change of variables formulas
\begin{equation}
  \int_{\omega^l} g^l\Gamma
  =  \int_{\omega} g |B|_d\Gamma_h,
\end{equation}
with $|B|_d$ denoting the absolute value of the determinant
of $B$.
The determinant $|B|_d$ satisfies the following estimate
\begin{lemma}  It holds
  \label{lem:detBbounds}
  \begin{alignat}{5}
    \| 1 - |B|_d \|_{L^\infty(\Gammah)}
    &\lesssim h^2,
    & &\qquad
    \||B|_d\|_{L^\infty(\Gammah)}
    &\lesssim 1,
    & &\qquad
    \||B|_d^{-1}\|_{L^\infty(\Gammah)}
    &\lesssim 1.
    \label{eq:detBbound}
  \end{alignat}
\end{lemma}
\begin{proof}
  See \cite{DemlowDziuk2007,Demlow2009} and \cite{BurmanHansboLarson2015}
  for a proof in the case of $d = k-1$.
  Recall that given a Riemannian metric on a $d$-dimensional manifold $\Gamma$,
  the canonical measure $\dGamma$ on $\Gamma$ is defined
  by the unique volume form $\omega_{\Gamma}$ satisfying $\omega_{\Gamma}(e_1,\ldots,e_k) = 1$ for
  one (and hence any) orthonormal
  frame $\{e_i\}_{i=1}^d$ on the tangential bundle $T\Gamma$.
  Writing $\dGamma = \omega_{\Gamma}$,
  the volume form is given by the pullback
  \begin{align}
    \dGamma = i^{\ast}(e^1 \wedge \ldots \wedge e^{d}),
  \end{align}
  i.e., the restriction to $\Gamma$ of the $d$-form defined by the outer product of the
  dual coframe $\{e^i\}_{i=1}^{d}$ satisfying $e^i(e_j) = \mean{e_i, e_j} = \delta^i_j$. Here, $i : \Gamma \hookrightarrow \RR^k$
  denotes the inclusion of $\Gamma$ into $\RR^k$ given by the identity map.
  Thanks to the evaluation formula
  \begin{align}
    (e^1 \wedge \ldots \wedge e^{d}) (v_1,\ldots, v_k)
    = \det((e^i(v_j))
    = \det(\mean{e_i,v_j}),
  \end{align}
  the defined form $\dGamma$ clearly satisfies $\dGamma(e_1,\ldots, e_d) = 1$.
  Now the pull-backed volume form $p^{\ast} \dGamma$ is described in terms of the volume form $\dGammah$
  by the identity
  $p^{\ast} \dGamma = |B|_d \dGammah$, where $|B|_d$ is determinant of $B$ as a linear mapping $ B : T_xK \to T_{p(x)}\Gamma$ and $\dGammah$ denotes the canonical volume form defined on $\Gammah$.
  Thus we have the transformation rule
  $ \int_{K^l} f \dGamma = \int_K p^{\ast}(f\dGamma)  = \int_K  f^e |B|_d \dGammah$.
  Taking an orthonormal tangential frame $\{e_i^h\}_{i=1}^d$  of $T\Gamma_h$,
  the determinant $|B|_d$ can be simply computed to
  \begin{align}
    |B|_d = p^{\ast}\dGamma(e_1^h,\ldots,e_d^h)
    = \dGamma(Dpe_1^h,\ldots,Dpe_d^h)
    = \det(\mean{e_i, Dpe_j^h}).
  \end{align}
  Next, observe that the representation~\eqref{eq:dp-representation} of $Dp$ yields
  $
  \mean{e_i, Dpe_j^h}
  =
  \mean{e_i, \Ps e_j^h} + O(h^2)
  =
  \mean{e_i, e_j^h} + O(h^2)
  $. Moreover, for $i = j$, one has $2(1 - \mean{e_i, e_i^h}) = \mean{ e_i - e_i^h, e_i - e_i^h}
  \lesssim h^2$
  while for $i\neq j$, $\mean{e_i, e_j^h} = \mean{e_i,e_j^h - e_j} \lesssim h$.
  Consequently,
  \begin{equation}
    \det(\mean{e_i, Dpe_j^h})
    = \det(a_{ij})
    \qquad \text{with }
    a_{ij} =
    \begin{cases}
      1 + O(h^2) &\quad \text{if } i = j, \\
      O(h)       &\quad \text{else}.
    \end{cases}
  \end{equation}
  Recalling the definition of the determinant
  $
  \det(a_{ij}) = \sum_{\sigma \in S(d)} \sig(\sigma) \Pi_{i=1}^{d} a_{i\sigma(i)}
  $
and examining the product for a single permutation $\sigma \in S(d)$
we see that
\begin{equation}
  \Pi_{i=1}^{d} a_{i\sigma(i)} =
  \begin{cases}
    (1 + O(h^2))^d &\quad \text{if } \sigma = \mathrm{Id}, \\
    O(h^2)  &\quad \text{else},
  \end{cases}
\end{equation}
since any other permutation than the identity involves at least two non-diagonal elements.
Hence  $|B|_{d} = 1 + O(h^2)$.
\end{proof}
We conclude this section by noting that combining the
estimates~\eqref{eq:BBTbound} and \eqref{eq:detBbound}
for respectively the norm and the determinant of $B$
shows that for $m = 0,1$
\begin{alignat}{3}
  \| v \|_{H^{m}(\mcK_h^l)} &\sim \| v^e \|_{H^{m}(\mcK_h)}
  & &\quad \text{for } v \in H^m(\mcK_h^l),
  \label{eq:norm-equivalences-ve}
  \\
  \| w^l \|_{H^{m}(\mcK_h^l)} &\sim \| w \|_{H^{m}(\mcK_h)}
  & &\quad \text{for } w \in V_h.
  \label{eq:norm-equivalences-wh}
\end{alignat}


\subsection{Fat Intersection Covering}
\label{ssec:fat-intersection-covering}
Since the manifold geometry is embedded into a fixed background mesh, the
active  mesh $\mcT_h$ might contain elements which barely intersect the
discretized manifold $\Gamma_h$.
Such ``small cut elements'' typically prohibit the application of a whole set of
well-known estimates, such as interpolation estimates and inverse inequalities,
which typically rely on certain scaling properties.
As a partial replacement for the lost scaling properties we here recall from~\cite{BurmanHansboLarson2015}
the concept of \emph{fat intersection coverings} of $\mcT_h$.

In \citet{BurmanHansboLarson2015} it was proved
that the active  mesh fulfills a fat intersection property which
roughly states that for every element there is a
close-by element which has a significant intersection with $\Gamma_h$.
More precisely, let $x$ be a point on $\Gamma$ and let
  $B_{\delta}(x) = \{y\in \RR^d: |x-y| < \delta\}$
  and
  $D_{\delta} = B_{\delta}(x) \cap \Gamma$.
  We define the sets of elements
  \begin{align}
    \mcK_{\delta,x}
    = \{ K \in \mcK_h : \overline{K}^l \cap D_{\delta}(x) \neq
    \emptyset \},
    \qquad
    \mcT_{\delta,x}
    = \{ T \in \mcT_h : T \cap \Gamma_h \in \mcK_{\delta,x} \}.
    \label{eq:fat-intersection-covering}
  \end{align}
 With $\delta \sim h$ we use the notation $\mcK_{h,x}$ and
$\mcT_{h,x}$. For each $\mcT_h$, $h \in (0,h_0]$ there is a set of
points $\mcX_h$ on $\Gamma$ such that $\{\mcK_{h,x}, x \in \mcX_h \}$
and $\{\mcT_{h,x}, x \in \mcX_h \}$ are coverings of $\mcT_h$ and
$\mcK_h$ with the following properties:
\begin{itemize}
  \item The number of set containing a given point $y$ is uniformly
    bounded
    \begin{align}
      \# \{ x \in \mcX_h : y \in \mcT_{h,x}  \} \lesssim 1 \quad
      \foralls y \in \RR^k
    \end{align}
    for all $h \in (0,h_0]$ with $h_0$ small enough.
  \item
    The number of elements in the sets $\mcT_{h,x}$ is uniformly bounded
    \begin{align}
      \# \mcT_{h,x} \lesssim 1
      \quad \foralls x \in \mcX_h
    \end{align}
    for all $h \in (0,h_0]$ with $h_0$ small enough, and each element in
    $\mcT_{h,x}$ shares at least one face with another element in
    $\mcT_{h,x}$.
  \item $\foralls h \in (0,h_0]$ with $h_0$ small enough, and $\foralls x \in \mcX_h$, $\exists
    T_x \in \mcT_{h,x}$ that has a large (fat) intersection with
    $\Gamma_h$ in the sense that
    \begin{align}
      | T_x | \sim h^{c} | T_x \cap \Gamma_h | = h^{c}| K_x |
      \quad \foralls x \in \mcX_h.
    \end{align}
\end{itemize}
While the proof in \cite{BurmanHansboLarson2015} was only concerned with the
surface case $d = k - 1$, it directly transfers to the case of
arbitrary codimensions.

\subsection{Fat Intersection Property for the Discrete Normal Tube}
\label{ssec:fat-intersection-property-discrete-normal-tube}
The goal of this section is to present a refined version
of the fat intersection covering, roughly stating that
a significant portion of each element can be reached from the discrete manifold
$\Gammah$ by walking along normal-like paths which reside completely inside $\mcT_h$.
We will need the following notation.
\begin{itemize}
  \item Let $T \in \mcTh$ and let $\mcN(T)\subset \mcTh$ denote the set
    of all neighbors to $T$ that also belongs
    to the active mesh $\mcTh$.
  \item Let $x$ be a vertex to an element $T \in \mcT_{h,0}$ then
    the star $\mcS(x)$ is the set of \emph{all} elements in the background
    mesh $\mcT_{h,0}$ that share the vertex $x$.
\end{itemize}
\begin{lemma}
  \label{lem:discrete-tangent-plane-approx-prop}
    For each $T\in \mcTh$ there is a $d$-dimensional plane $\barGamma = \barGamma_T$
    with constant normal bundle $\{\barn_i\}_{i=1}^c = \{\barn_{i,T}\}_{i=1}^c$
    satisfying the geometry approximation assumptions
    \begin{gather}
      \Gammah \cap \mcN(T)  \subset U_{\epsilon}(\barGamma),
      \qquad
      \sup_{T \in \mcTh} \| \barn_{i,T} - {n_i} \|_{L^{\infty}(\mcN(T))} \lesssim h
      \quad \text{for } i = 1,\ldots, c,
      \label{eq:discrete-tangent-plane-approx-prop}
    \end{gather} 
    with $\epsilon \sim h^2$.
    Furthermore, $\Gammah\cap \mcN(T)$ is a
    Lipschitz function over $\barGamma$ and its Lipschitz constant is
    uniformly bounded over all $T\in \mcTh$.
\end{lemma}
\begin{proof}
  To verify (\ref{eq:discrete-tangent-plane-approx-prop}) we take $x \in
  p(\mcN(T)) \subset \Gamma$ and let $\barGamma_T$ be the tangent plane to
  $\Gamma$ at $x$. Next we note that by the geometry approximation assumptions we
  have $\Gammah \subset U_{\delta^{'}}(\Gamma)$ with $\delta^{'} \sim h^2$.
  Now let $B_{\delta}(x)$ be a ball of radius $\delta \sim h$ such that
  $\mcN(T)\subset B_\delta(x)$ and $p(\mcN(T)) \subset B_\delta(x)$. Using
  the smoothness of $\Gamma$ and the fact that $\barGamma$ is the tangent plane
  to $\Gamma$ at $x$ we find that there is $\delta^{''}\sim h^2$ such that
  \begin{equation}
    \Gammah \cap \mcN(T)
    \subset
    \Gammah \cap B_\delta(x)
    \subset
    U_{\delta'}(\Gamma) \cap B_\delta(x)
    \subset U_{\delta^{''}}(\barGamma).
  \end{equation}
  Finally, we note that 
  choosing $\barn_i = n_i(x)$, the 
  smoothness assumptions on $\Gamma$ yields 
  $\|\barn_i - n_i(y)\|_{\IR^3} = \|n_i(x) - n_i(y)\|_{\IR^k}
  \lesssim \delta \sim h$ for $y \in B_\delta(x)$. 
\end{proof}
Next, we introduce some notation to describe normal-like paths given
by projecting sets into $\Gammah$.
  \begin{itemize}
    \item Let $\omega$ be a set and $x$ a point then the cone with base $\omega$
      and vertex $z$ is defined by
      \begin{equation}\label{eq:cone}
        \Cone(\omega,z) = \bigcup_{x \in \omega} I(z,x),
      \end{equation}
      where $I(z,x)$ is the line segment with endpoints $x$ and 
      $x$.
    \item
      Let $\barp_h$ be the mapping onto $\Gammah$
      obtained by following a unique normal direction $\barn \in \spann\{\barn_1, \ldots, \barn_c\}$
      from $x$ to $\Gammah$.
      Given a set $\omega$ we define the cylinder over $\Gammah$ by
      \begin{equation}\label{eq:cylinder}
        \Cyl (\omega,\Gamma_h) 
        = \bigcup_{x \in \omega} I(x,\barp_h(x)).
      \end{equation}
  \end{itemize}
  Then we can formulated the following Lemma.
\begin{lemma}
\label{lem:fat-intersection-property-discrete-normal-tube}
  For each $T \in \mcT_h$ there is a ball $B_{\delta} \subset T$
  with radius $\delta \sim h$
  such that
  \begin{align}
    \Cyl (B_{\delta}, \Gammah) \subset \mcN(T).
    \label{eq:fat-intersection-property-discrete-normal-tube}
  \end{align}
\end{lemma}
\begin{proof} 
  To keep the notation at a moderate level, we restrict ourselves to the
  most important case $c = 1$.
  For $T \in \mcT_h$, let $R_T$ be the radius of the circumscribed sphere of
  element $T$ and $r_T$ the radius of the inscribed sphere in $T$.  
  The center of the inscribed sphere is denoted by $x_T$.  
  We recall that the element is shape regular which means that $r_T\sim R_T
  \sim h$. Let $\{x_i\}_{i=0}^k$ be the vertices of $T$,
  then by shape regularity there is
  $\delta_1 \sim h$ such that ball $B_{\delta_1}(x_i) \subset
  \mcS(x_i)$ for each $i$. For technical reasons we will also choose
  $\delta_1$ such that
  \begin{equation}
    \delta_1 \leq \min_{i=\in \{0,\ldots, k\}} \| x_i - x_T\|_{\IR^k},
  \end{equation}
  where $x_T$ is the center of the inscribed sphere. By shape regularity it
  follows that $r_T \leq \delta_i \leq R_T$, $i=0,\ldots,k$ and thus we may
  still take $\delta_1 \gtrsim h$. To
  prove~\eqref{eq:fat-intersection-property-discrete-normal-tube}, we consider
  two different intersection cases.
  \\
  {\bf Intersection Case ${\bfI}$.}  
  Assume that
  \begin{equation}\label{eq:local-normal-grad}
    \barGamma \cap T \subset T \setminus \left( \bigcup_{i=0}^k B_{\delta_1/8}(x_i) \right),
  \end{equation}
  then we shall construct a ball $B_{\delta_2}(x) \subset T$ with
  $x \in \barGamma \cap T.$
  and $\delta_2 \sim h$.
  We note that $\barGamma$ must intersect at least one of the
  $k+1$ line segments $I(x_{T}, x_i)$ that join $x_T$ with the nodes
  $x_i$, say the line segment from $x_T$ to $x_0$, and that there is an intersection point $z=  I(x_{T}, x_0) \cap \Gamma_h$ such that
  \begin{equation}
    \delta_1/8 \leq \|z - x_0\|_{\IR^k},
  \end{equation}
  We note that  $\Cone (B_{r_T}(x_T), x_0) \subset T$ and that
  $B_{\delta_2}(z) \subset \Cone (B_{r_T}(x_T), x_0)$ where
  \begin{equation}
    \delta_2 =  r_T\frac{\|z - x_{T}\|_{\IR^k}}{\| x_0 - x_{T} \|_{\IR^k}}
  \end{equation}
  is a suitable scaling of $r_T$. We also note that $\delta_2\sim h$ since
  \begin{align}
    \delta_2 &= r_T\frac{\|z - x_{T}\|_{\IR^k}}{\| x_0 - x_{T} \|_{\IR^k}}
    \geq
    \frac{r_T}{2} \frac{\|x_0 - x_{T}\|_{\IR^k} - \delta_1/8+ 2\epsilon}{\| x_0 - x_{T} \|_{\IR^k}}
    \\
    &\geq
    \frac{r_T}{2} \left( 1 -\frac{\delta_1}{8 \underbrace{\| x_0 - x_{T} \|_{\IR^k}}_{\geq \delta_1}}\right)
    \geq
    r_T \frac{7}{16} \sim h.
  \end{align}
  We finally note that for $\epsilon/\delta_1$ small enough we clearly have
  \begin{equation}
    \Cyl(B_{\delta_2/2}(z),\Gammah) \subset B_{\delta_2}(z) \subset T.
  \end{equation}

  \paragraph{\emph{Intersection Case ${II}$}} There is at least one
  $i$, say $i=0$, such that
  \begin{equation}
    B_{\delta_1/8}(x_0) \cap \barGamma \neq \emptyset.
  \end{equation}
  We divide this case in two subcases
  \begin{equation}\label{eq:caseIIaandb}
    \begin{cases}
      B_{r_T/2}(x_T) \cap \barGamma \neq \emptyset &\quad \text{Case $II_1$},
      \\
      B_{r_T/2}(x_T) \cap \barGamma = \emptyset &\quad \text{Case $II_2$}.
    \end{cases}
  \end{equation}

\paragraph{Case ${II}_{1}$.} Let $z$ be the point
on $\barGamma$ with minimal distance to $x_T$, we then have
$\| z - x_T \|_{\IR^k} \leq r_T/2$, and $B_{r_T/2} (z) \subset
B_{r_T}(x_T)$. Then we conclude that
\begin{equation}
  \Cyl(B_{r_T/4}(z), \Gammah) \subset B_{r_T/2} (z) \subset T
\end{equation}
for $\epsilon/r_T$ small enough.

\paragraph{Case ${II}_{2}$.}
Consider the ball $B_{\delta_1}(x_0) \subset \mcS(x_i)$, and
observe that we have a partition
\begin{equation}
  B_{\delta_1}(x_0) = B^+_{\delta_1}(x_0)
  \cup (B_{\delta_1}(x_0)\cap \Gammah) \cup B^-_{\delta_1}(x_0),
\end{equation}
where $B^\pm_{\delta_1}(x_0)$ are the two connected components of
$B^+_{\delta_1}(x_0)\setminus \Gammah$. Without loss of generality we
may assume that $x_0 \in (B_{\delta_1}(x_0)\cap \Gammah) \cup B^-_{\delta_1}(x_0)$. Then we have
\begin{equation}\label{eq:caseII-aa}
  B^+_{\delta_1}(x_0) \subset \mcS(x_0) \cap \mcN(T) \subset \mcN(T).
\end{equation}
To verify (\ref{eq:caseII-aa}) we note that if $x_T \in \Gammah$ we have
$B^+_{\delta_1}(x_0) \subset B_{\delta_1}(x_0) \subset \mcS(x_0)
\subset \mcN(T)$. Next, if $x_0 \in B^-_{\delta_1}(x_0)$ we instead note
that an element $T'$ in $\mcS(x_0)$ which do not belong to $\mcTh$ must
satisfy $T'\cap B_{\delta_1}(x_0) \subset B^-_{\delta_1}(x_0)$ and thus
we conclude that all elements in $\mcS(x_0)$ that intersect
$B^+_{\delta_1}(x_0)$ must be in the active mesh $\mcTh$. We therefore
have $B^+_{\delta_1}(x_0) \subset \mcS(x_0) \cap \mcTh$ which concludes
the verification of (\ref{eq:caseII-aa}).

Next we note that it follows from the assumptions in Case $II_1$
that for $\epsilon/r_T$ small enough it holds $B_{r_T/2-2\epsilon} \cap
U_\epsilon(\barGamma) = \emptyset$, $B_{r_T/4}(x_T)\subset B_{r_T/2-2\epsilon}$
and thus in particular $B_{r_T/4}(x_T)\cap U_\epsilon(\barGamma)=\emptyset$,
and
\begin{equation}
  \left(B_{\delta_1}(x_0) \setminus B_{\delta_1/8+\epsilon}(x_0)\right)
  \cap \Cone ( B_{r_T/4}( x_T ), x_0 )
  \subset B^+_{\delta_1}(x_0).
\end{equation}
For $\epsilon/\delta_1$ small enough we have $B_{\delta_1/8+\epsilon}(x_0)
\subset B_{\delta_1/4}(x_0)$ and may in the same way as in
Case $I$ construct a ball
$B_{\delta_3}(z)$, with $z$ on $I(x_0,x_T)$ and $\delta_3 \sim h$
such that
\begin{equation}
  B_{\delta_3}(z)
  \subset\left(B_{\delta_1/2}(x_0) \setminus
  B_{\delta_1/4}(x_0) \right)\cap \Cone ( B_{r_T/4}( x_T ), x_0 )
  \subset
  B^+_{\delta_1}(x_0).
\end{equation}
Then, for $\epsilon/\delta_1$ small enough the cylinder
$\Cyl(B_{\delta_3},\Gamma_h)$ satisfies
\begin{equation}\label{eq:claimII-bb}
  \Cyl(B_{\delta_3},\Gamma_h) \subset B^+_{\delta_1}(x_0)
  \subset \mcN(T),
\end{equation}
which concludes the proof.
\end{proof}

\section{Verification of the Inverse and Discrete Poincar\'e Estimates}
\label{sec:poincare-estimate-verification}
We now show that any combination of
discrete bilinear forms $a_h$ and stabilization forms $s_h$
from Table~\ref{tab:cutfem-form-realizations} in Section~\ref{sec:cutfem-realizations}
yields a stabilized cut finite element
formulation which satisfies both the discrete Poincar\'e
estimate~(\ref{eq:discrete-poincare-Ah-abstract}) and the inverse
estimate~(\ref{eq:inverse-estimate-Ah-abstract}).
The core idea behind the forthcoming proofs of the discrete Poincar\'e estimates
is to show that a properly scaled $L^2$ norm of a discrete function $v \in V_h$
computed on $\mcT_h$ can be controlled by the $L^2$ norm on the discrete manifold $\Gamma_h$
augmented by the stabilization form in question,
\begin{align}
  \| v \|_{\mcT_h}^2 \lesssim h^c \| v \|_{\Gammah}^2 + s_h(v,v) \quad \foralls v \in \mcV_h.
  \label{eq:l2norm-control-via-sh}
\end{align}
Then the desired Poincar\'e estimate follows directly from estimating
$\| v \|_{\Gammah}$ using the $\Gammah$-based discrete Poincar\'e inequality
stated and proved in \cite[Lemma 4.1]{BurmanHansboLarson2015}:
\begin{lemma}
  \label{lem:poincare-I} For $v \in V_h$, the following estimate holds
  \begin{equation}
    \| v - \lambda_{\Gamma_h}(v) \|_{\Gamma_h}
    \lesssim
    \| \nablash v \|_{\Gamma_h}
    \label{eq:poincare-I}
  \end{equation}
  for $0<h \leq h_0$ with $h_0$ small enough.
\end{lemma}
The forthcoming proofs of ~\eqref{eq:l2norm-control-via-sh}
will also make use of various inverse estimates which we state first.

\subsection{Inverse Estimates}
\label{ssec:inverse-estimates}
Recall that for given
$T \in \mcT_h$, the following
well-known inverse estimates hold for $v_h \in V_h$:
\begin{gather}
  \label{eq:inverse-estimates-standard}
  \| \nabla v_h\|_{T}
  \lesssim
  h^{-1}
  \| v_h\|_{T},
  \qquad
  \| v_h\|_{\partial T}
  \lesssim
  h^{-1/2}
  \| v_h\|_{T},
  \qquad
  \| \nabla v_h\|_{\partial T}
  \lesssim
  h^{-1/2}
  \| \nabla v_h\|_{T}
\end{gather}
In addition, in the course of our analysis, we will also employ ``cut versions''
of these inverse estimates as specified in
\begin{lemma}
  Let $K \in \mcK_h$ and  $T \in \mcT_h$, then
  \begin{alignat}{5}
    \label{eq:inverse-estimate-cut-v-on-K}
    h^{c}\|v_h \|_{K \cap T}^2
    &\lesssim
    \|v_h\|_{T}^2,
    & & \qquad
    h^{c}\| \nabla v_h \|_{K \cap T}^2
    &\lesssim
    \|\nabla v_h\|_{T}^2
  \end{alignat}
\end{lemma}
\begin{proof}
  Recalling that the mesh is supposed to be shape regular and
  labeling finite element functions and sets defined on the standard reference
  element with $\widehat{\cdot}$, the proof follows immediately from a
  standard scaling and finite dimensionality argument leading to the following chain of
  estimates
  \begin{align}
    \| v_h \|_{\Gamma \cap T}^2
    \lesssim
    h^d \| \widehat{v}_h \|_{\widehat{\Gamma} \cap \widehat{T}}^2
    \lesssim
    h^d \| \widehat{v}_h \|_{L^{\infty}(\widehat{T})}^2
    \underbrace{
      |\widehat{\Gamma} \cap \widehat{T}|
    }_{\lesssim 1}
    \lesssim
    h^d \| \widehat{v}_h \|_{\widehat{T}}^2
    \lesssim
    h^{d-k} \| v_h \|_{{T}}^2
  \end{align}
  which is precisely the first inequality in
  \eqref{eq:inverse-estimate-cut-v-on-K}. The second one can be
  derived analogously.
\end{proof}
Now the verification of the abstract inverse estimate~(\ref{eq:inverse-estimate-Ah-abstract})
for any combination of discrete bilinear forms $a_h^i$ and stabilizations $s_h^i$ for $i=1,2$
is a simple consequence of the following lemma.
\begin{lemma} Let $v \in V_{h,0}$ then the following inverse estimate holds
  \label{lem:inverse-estimate-Ah-concrete}
  \begin{align}
    \| v \|_{a_h^i}^2 \lesssim h^{-2-c} \| v \|_{\mcT_h}^2 \quad i = 1,2
    \label{eq:inverse-estimate-ah}
  \end{align}
  \begin{align}
    \| v \|_{s_h^1}^2 \lesssim h^{-2-c} \| v \|_{\mcT_h}^2,
    \qquad
    \| v \|_{s_h^2}^2 \lesssim h^{-c} \| v \|_{\mcT_h}^2,
    \qquad
    \| v \|_{s_h^3}^2 \lesssim h^{-2-c+\alpha} \| v \|_{\mcT_h}^2,
    \label{eq:inverse-estimate-sh}
  \end{align}
\end{lemma}
  \begin{proof}
    Since $\| v \|_{a_h^1} \leqslant \| v \|_{a_h^2}$,
    the proof of (\ref{eq:inverse-estimate-ah}) follows directly
    from combining the second estimate from~(\ref{eq:inverse-estimate-cut-v-on-K})
    with the first standard inverse estimate in~(\ref{eq:inverse-estimates-standard}).
    Next, successively applying the
    last and first inverse estimate recalled in
    ~\eqref{eq:inverse-estimates-standard} shows that for $s_h^1$ and $s_h^2$
    \begin{align}
      s_h^1(v,v) &= h^{1-c}\|\jump{n_F \cdot \nabla v}\|_{\mcF_h}
      \lesssim h^{-c} \| \nabla v \|_{\mcT_h}
      = h^{-2} s_h^2(v, v) \|_{\mcT_h}
      \lesssim h^{-2-c} \| v \|_{\mcT_h}.
    \end{align}
      Similarly,
    \begin{align}
      s_h^3(v,v) &= h^{\alpha-c} \| \Qsh \nabla v \|_{\mcT_h}^2
                 \lesssim
                 h^{-2-c+\alpha} \| v \|_{\mcT_h}^2,
    \end{align}
    which concludes the proof.
  \end{proof}

\subsection{Analysis of the Face-based Stabilization $s_h^1$}
\label{ssec:analysis-face-based-stab}
The analysis of the face-based stabilization 
was presented in full detail in \cite{BurmanHansboLarson2015,BurmanHansboLarsonEtAl2016a} 
in the case of codimension $c = 1$.
Here, we only note that the proof literally transfers to the general case 
$c >1$ after replacing the inverse estimates and fat intersection properties
stated in \cite{BurmanHansboLarson2015,BurmanHansboLarsonEtAl2016a} by their
properly scaled equivalents introduced in
Section~\ref{ssec:fat-intersection-covering} and
Section~\ref{ssec:inverse-estimates}.
For completeness, we state the final discrete Poincar\'e estimate.
\begin{lemma}
  \label{lem:discrete-poincare-Nh-face-stab} For $v \in V_h$ and
  $0<h \leq h_0$ with $h_0$ small enough, it holds
  \begin{equation}
    \| v - \lambda_{\Gamma_h}(v) \|^2_{\mcT_h}
    \lesssim
    h^{c}\| \nablash v \|_{\Gamma_h}^2
    + h^{1-c} \| \jump{n_F \cdot \nabla v} \|_{\mcF_h}^2.
    \label{eq:discrete-poincare-Nh-face-stab}
  \end{equation}
  In particular, for $i=1,2$,
  $\| \cdot \|_{a_h^i} + \| \cdot \|_{s_h^1}$  defines a norm
  for $v \in V_{h,0}$:
  \begin{equation}
    h^{-c} \| v \|^2_{\mcT_h}
    \lesssim
    \| v \|_{a_h^i}^2 + \|v\|_{s_h^1} \quad \text{for } i=1,2.
    \label{eq:discrete-poincare-Ah-face-stab}
  \end{equation}
\end{lemma}

\subsection{Analysis of the Full Gradient Stabilization $s_h^2$}
\label{ssec:analysis-full-grad-stab}
We start with the following  lemma
which describes how the control of discrete
functions on potentially small cut elements can be transferred to
their close-by neighbors with large intersections by using the full gradient stabilization term.
\begin{lemma}
  \label{lem:l2norm-control-via-full-gradient}
  Let $v \in V_h$
  and consider a macro-element $\mathcal{M} = T_1 \cup T_2$ formed by
  any two elements $T_1$ and $T_2$ of $\mcT_h$ which share at least a vertex.
  Then
  \begin{align}
    \| v \|_{T_1}^2 &\lesssim  \|v\|_{T_2}^2
    + h^2 \| \nabla v\|_{T_1}^2,
    \label{eq:l2norm-control-via-full-gradient}
  \end{align}
  with the hidden constant only depending on the quasi-uniformness parameter.
\end{lemma}
\begin{proof}
  Let $x_0$ be a vertex shared by $T_1$ and $T_2$ and denote
  by $v_i = v|_{T_i}$ the restriction of $v$ to $T_i$. Since
  $v_i$ is linear,
  \begin{align}
    v_i(x) =  v_i(x_0) +(x-x_0) \nabla v_i(x)
  \end{align}
  and consequently, using a Young inequality and
  the fact the shape regularity implies
  $|T| \sim h^k$
  and
  $\|x-x_0\|_{L^{\infty}(T)} \lesssim h$,
  we see that
  \begin{align}
    \|v_1\|_{T_1}^2
    \lesssim
    h^k |v_1(x_0)| + h^2 \| \nabla v_1 \|_{T_1}^2
    \lesssim
    \|v_2 \|_{T_2}^2 + h^2 \| \nabla v_1 \|_{T_1}^2,
  \end{align}
  where we used that $v_1(x_0) = v_2(x_0)$ and an inverse inequality
  of the form $h^k v_2(x_0) \lesssim \|v_2\|_{T_2}^2$.
\end{proof}
Now the fat intersection property from Section~\ref{ssec:fat-intersection-covering}
guarantees that
Lemma~\ref{lem:l2norm-control-via-full-gradient} only
needs to be applied a bounded number of times to transfer the $L^2$
control from an arbitrary element to an element with a fat
intersection. On an element with a fat intersection
$ h^c | T \cap \Gamma_h | \sim | T |$, the control of the $L^2$ norm
can be passed -- via piecewise constant approximations
of $v$ -- from the element to the discrete manifold part,
where $v \in V_h$ satisfies the
Poincar\'e inequality\eqref{eq:poincare-I} on the surface.
More precisely, we have the following discrete Poincar\'e inequality,
which involves a scaled version of the $L^2$ norm
of discrete finite element functions
on the active mesh.
\begin{lemma}
  \label{lem:discrete-poincare-Nh-full-grad} For $v \in V_h$, the following estimate holds
  \begin{equation}
    h^{-c} \| v - \lambda_{\Gamma_h}(v) \|^2_{\mcT_h}
    \lesssim
    \| \nablash v \|_{\Gamma_h}^2
    + h^{2-c} \| \nabla v \|_{\mcT_h}^2
    \label{eq:discrete-poincare-Nh-full-grad}
  \end{equation}
  for $0<h \leq h_0$ with $h_0$ small enough.
\end{lemma}
\begin{proof}
  Without loss of generality we can assume that $\lambda_{\Gamma_h}(v) = 0$.
  After applying~\eqref{eq:l2norm-control-via-full-gradient}
  \begin{align}
    \| v \|_{\mcT_h}^2
    &\lesssim \sum_{x \in \mcX_h} \| v \|_{\mcT_{h,x}}^2
    \lesssim
    \sum_{x \in \mcX_h} \| v \|_{T_x}^2
    + h^2 \| \nabla v \|_{\mcT_h}^2,
    \label{eq:poincare-proof-I}
  \end{align}
  it is sufficient to proceed with the first term
  in~\eqref{eq:poincare-proof-I}.
  For $v \in V_h$, we define a piecewise constant approximation satisfying
  $\|v - \overline{v}\|_T \lesssim h \| \nabla v \|_T$, e.g. by taking $\overline{v} = v(x_0)$ for any point $x_0 \in T$.
  Adding and subtracting $\overline{v}$ gives
  \begin{align}
    \sum_{x \in \mcX_h} \| v \|_{T_x}^2
    &\lesssim
    \sum_{x \in \mcX_h}  \| v - \overline{v} \|_{T_x}^2
    +
    \sum_{x \in \mcX_h}  \| \overline{v} \|_{T_x}^2
    \\
    &\lesssim
    h^2 \| \nabla v \|_{\mcT_h}^2
    + \sum_{x \in \mcX_h} h^c \| \overline{v} \|_{K_x}^2
    \\
    \label{eq:transition-to-surface-I}
    &\lesssim
    h^2 \| \nabla v \|_{\mcT_h}^2
    + h^c \| v \|_{\Gamma_h}^2
    +
    \underbrace{h^c \| v - \overline{v} \|_{\Gamma_h}^2}_{\lesssim h^2 \|
    \nabla v\|_{\mcT_h}^2}
    \\
    &\lesssim
    h^2 \| \nabla v \|_{\mcT_h}^2
    + h^c \| \nablash v \|_{\Gamma_h}^2,
    \label{eq:transition-to-surface-II}
  \end{align}
  where the inverse inequality~(\ref{eq:inverse-estimate-cut-v-on-K})
  was used in ~\eqref{eq:transition-to-surface-I}
  to find that
  $
  h^c \| v - \overline{v} \|_{K_x}^2
  \lesssim
  \| v - \overline{v} \|_{T_x}^2
  \lesssim
  h^2\| \nabla \|_{T_x}
  $,
  followed by an application of the Poincar\'e
  inequality~\eqref{eq:poincare-I}
  in the last step.
\end{proof}
\begin{figure}[htb!]
  \begin{subfigure}[t]{0.45\textwidth}
    \includegraphics[page=1,width=1\textwidth]{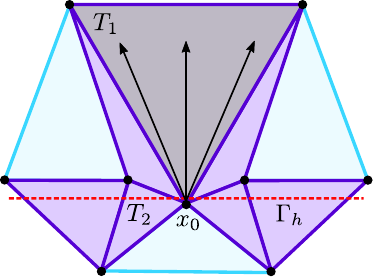}
  \end{subfigure}
  \hspace{0.04\textwidth}
    \begin{subfigure}[t]{0.45\textwidth}
    \includegraphics[page=1,width=1\textwidth]{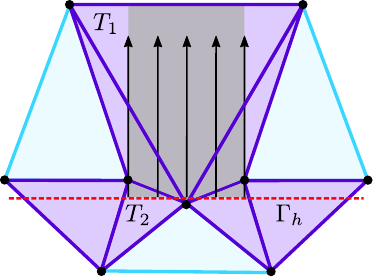}
    \end{subfigure}
    \caption{Fat intersection properties and $L^2$ control mechanisms for the full gradient and normal gradient
      stabilization. 
      (Left) While element $T_1$ has only a small intersection with $\Gamma_h$, there are several neighbor elements
      in $\mcT_h$ which share the node $x_0$ and have a fat intersection with $\Gamma_h$. The appearance of the
      full gradient in stabilization $s_h^2$ allows to integrate along arbitrary directions and
      thus gives raise to the control of $\| v \|_{T_1}$ through Lemma~\ref{lem:l2norm-control-via-full-gradient}.
      (Right) The fat intersection property for the discrete ``normal'' tube guarantees that
      still a significant portion of $T_1$ can be reached when integrating along normal-like paths
      which start from $\Gamma_h$ and which reside completely inside $\mcT_h$.
    }
    \label{fig:condition_number-example}
\end{figure}

\subsection{Analysis of the Normal Gradient Stabilization $s_h^3$}
\label{ssec:analysis-normal-grad-stab} The goal of this section is to
prove the discrete Poincar\'e
inequality~(\ref{eq:discrete-poincare-Ah-abstract}) by establishing
inequality~(\ref{eq:l2norm-control-via-sh}) for the normal gradient
stabilization $s_h^3(v, w) = h^{\alpha - c} (\Qsh \nabla v, \Qsh
\nabla w)_{\mcT_h}$ with $\alpha \in [0,2]$.
Recalling the notation from
Section~\ref{ssec:fat-intersection-property-discrete-normal-tube},
we start with proving a local variant of~(\ref{eq:poincare-I})
which involves the normal projection 
$\barQ = \sum_{i=1}^c \barn_i \otimes \barn_i$
defined by the normal bundle $\{\barn_i\}_{i=1}^c$ 
associated with the local $d$-dimensional plane $\barGamma$ which approximates
$\Gamma_h$ as specified in Lemma~\ref{lem:discrete-poincare-Nh-face-stab}.
\begin{lemma}
  For $v\in V_h$ and $T\in \mcTh$, it holds
  \begin{equation}\label{eq:claim-a-a}
    \| v \|^2_T \lesssim h^c \| v \|^2_{\mcN(T)\cap \Gammah}
    + h^2 \| \barQ \nabla v \|^2_{\mcN(T)},
  \end{equation}
  where the hidden constant depends only on quasi-uniformness
  parameters.
\end{lemma}
\begin{proof}
  By the fat intersection
  property~\eqref{eq:fat-intersection-property-discrete-normal-tube}, there is
  for each $T \in \mcT_h$ a ball $B_{\delta}\subset T$ 
  with center $x_c$ and radius $\delta \sim h$
  such that $\Cyl(B_{\delta}, \Gammah \subset \mcN(T))$.
  Let $\barGamma_2$ be the $d$-plane 
  parallel with $\Gamma$ and passing through the center $x_c$.
  Then taking $\delta' = \sqrt{k} \delta$, the cubular neighborhood
  $Q_{\delta'} := Q_{\delta'}(\barGamma_2 \cap B_{\delta})$ associated with
  $\barGamma_2$ and its normal bundle $\{\barn_i\}_{i=1}^c$
  satisfies
  $Q_{\delta'}\subset B_{\delta}$.
  Since $|Q_{\delta'}| \sim |T|$, a finite dimensionality argument shows that
  \begin{align}
    \| v \|_{T}^2 
    \lesssim 
    \|v\|_{Q_{\delta'}}^2 \quad \foralls v \in V_h.  
  \end{align}
  Next, we apply the scaled Poincar\'e inequality~\eqref{eq:cubular-poincare-estimate} recursively with $\delta' \sim h$ to obtain
  \begin{align}
    \| v \|_T^2
    &\lesssim
      \| v \|_{Q^c_{\delta'}}^2
      \\
    &\lesssim
      h \| v \|_{Q^{c-1}_{\delta'}}^2
      +
      h^2 \| \barn_c \cdot \nabla v \|_{Q^{c}_{\delta'}}^2
    \\
    &\lesssim
      h
      \bigl( h \| v \|_{Q^{c-2}_{\delta'}}^2
      +
      h^2 \| \barn_{c-1}\cdot  \nabla v \|_{Q^{c-1}_{\delta'}}^2
      \bigl)
      + h^2 \| \barn_c \cdot \nabla v \|_{Q^{c}_{\delta'}}^2
    \\
    &\lesssim
      h^c \| v \|_{Q^{0}_{\delta'}}^2
      +
      h^2 \sum_{i=1}^c h^{c-i} \| \barn_i \cdot \nabla v \|_{Q^{i}_{\delta'}}^2
    \\
    &\lesssim
      h^c \| v \|_{\barGamma_2 \cap B_{\delta'}}^2
      +
      h^2 \| \barQ \nabla v \|_{T}^2,
    \label{pr:eq:pc-est-normal-grad-step-1}
  \end{align}
  where in the last step we used the inverse inequality
  $\| \barn_i \cdot \nabla v \|_{Q^{i}_{\delta'}}^2
  \lesssim
  h^{-(c - i)} \| \barn_i \cdot  \nabla v \|_{T}^2
  $
  which can be proven exactly as the inverse inequalities~(\ref{eq:inverse-estimate-cut-v-on-K}).
  It remains to estimate the first term in~\eqref{pr:eq:pc-est-normal-grad-step-1}.
  Recalling the definitions from
  Section~\ref{ssec:fat-intersection-property-discrete-normal-tube},
  we have the representation formula
  \begin{align}
    v(x) = v(\barp_h(x)) + \int_0^{\barrho_h(x)} \barn \cdot \nabla v (\barp_h(x) + s \barn) ds
  \end{align}
  for each $x \in \barGamma_2 \cap B_{\delta}$ since $\Cyl(\barGamma_2 \cap B_{\delta}) \subset \mcN(T)$.
  Here, $\barrho_h(x)$ is the distance $\|x - \barp_h(x)\|_{\IR^k}$
  satisfying $\barrho_h(x) \sim h$,
  and $\barn$ is the unit normal vector corresponding to $x - \barp_h(x)$.
  As before, we deduce that
  \begin{align}
    |v(x)|^2
    \lesssim
    | v(\barrho_h(x)|^2 +
    h \int_0^{\barrho_h(x)} |\barn \cdot \nabla v (\barp_h(x) + s \barn)|^2 ds.
  \end{align}
  After integrating over $\barGamma_2 \cap B_{\delta}$ we get
\begin{align}
  \int_{\barGamma_2 \cap B_{\delta}} v^2 \,\mathrm{d}\barGamma_2(x)
  &\lesssim
    \int_{\barGamma_2 \cap B_{\delta}} (v\circ \barp_h(x))^2
    \,\mathrm{d}\barGamma_2(x)
  +
    h \int_{\barGamma_2 \cap B_{\delta}}  \int_0^{\barrho_h(x)} |\barn \cdot \nabla v (\barp_h(x) + s \barn)|^2 ds
    \,\mathrm{d}\barGamma_2(x)
  \\
  &\lesssim
    \int_{\Gamma_h \cap \mcN(T)} v(x)^2
    \dGammah(x)
  +
    h \int_{\Gammah \cap \mcN(T)}  \int_0^{\barrho_h(x)} |\barn \cdot \nabla v (x + s \barn)|^2 ds
    \dGammah(x).
    \label{pr:eq:pc-est-normal-grad-step-2}
\end{align}
Observe that the last term
in~\eqref{pr:eq:pc-est-normal-grad-step-2} is the integral of the
discrete function $|\barn \cdot \nabla v (x + s \barn)|^2$
over a $(c-1)$-codimensional subset of $\Cyl(\barGamma_2 \cap
B_{\delta}) \subset \mcN(T)$, and thus an inverse inequality similar
to~\eqref{eq:inverse-estimate-cut-v-on-K} gives
\begin{align}
    h \int_{\Gammah \cap \mcN(T)}  \int_0^{\barrho_h(x)} |\barn \cdot \nabla v (x + s \barn)|^2 ds
    \dGammah(x)
    \lesssim
    h \cdot h^{-c+1} 
    \| \barn \cdot \nabla v \|_{\mcN(T)}^2,
\end{align}
and therefore
\begin{align}
  \| v \|_{\barGamma_2 \cap B_{\delta}}^2
  &\lesssim
  \| v \|_{\Gammah \cap \mcN(T)}^2 
  + h^{2-c} \| \barQ \nabla v \|_{\mcN(T)}^2.
  \label{pr:eq:pc-est-normal-grad-step-3}
\end{align}
Now inserting~\eqref{pr:eq:pc-est-normal-grad-step-3}
into~\eqref{pr:eq:pc-est-normal-grad-step-1} concludes the proof.
\end{proof}
Thanks to the geometric approximation properties~\eqref{eq:discrete-tangent-plane-approx-prop},
the previous lemma gives us the leverage to prove the main result
of this section.
\begin{proposition}
  \label{prop:normalgradstab}
  Assume that $v \in V_h$. Then 
  \begin{align}
    \label{eq:scaled-l2-norm-gradstab}
    h^{-c}\| v \|^2_\mcTh &\lesssim \| v \|^2_{\Gammah}
    + s_h^3(v,v)
    \\
    \label{eq:poincaregradstab}
    h^{-c}\| v  - \lambda_{\Gammah}(v)\|^2_\mcTh
    &\lesssim \| \nablash v \|^2_{\Gammah}
    + s_h^3(v,v)
  \end{align}
\end{proposition}
\begin{proof}
  \begin{align}
    \| v \|_{\mcT_h}^2
    = \sum_{T \in \mcT_h} \| v \|_{T}^2
    &\lesssim
      \sum_{T \in \mcT_h}
    \bigl(
    h^c \| v \|_{T\cap\Gammah}^2 + h^2\| \barQ \nabla v \|_{\mcN(T)}^2
    \bigr)
    \\
    &\lesssim
      h^c \| v \|_{\Gammah}
      +
      \sum_{T \in \mcT_h}
      h^2 \bigl(
      \| \Qsh \nabla v \|_{\mcN(T)}^2
      +
      \| (\Qsh - \barQ) \nabla v \|_{\mcN(T)}^2
    \bigr)
    \\
    &\lesssim
      h^c \| v \|_{\Gammah}
      +
      h^2 \| \Qsh \nabla v \|_{\mcT_h}^2
      +
      h^4 \| \nabla v \|_{\mcT_h}^2
    \\
    &\lesssim
      h^c \| v \|_{\Gammah}
      +
      h^2 \| \Qsh \nabla v \|_{\mcT_h}^2
      +
      h^2 \| v \|_{\mcT_h}^2
      \label{eq:poincare-est-s3-global}
  \end{align}
  For $h$ small enough, the last term in (\ref{eq:poincare-est-s3-global}) can be kick-backed
  to the left-hand side and as a result
  \begin{align}
    h^{-c}\| v \|_{\mcT_h}^2
    \lesssim
      \| v \|_{\Gammah}
      +
      h^{2-c} \| \Qsh \nabla v \|_{\mcT_h}^2
    =
      \| v \|_{\Gammah}^2 + s_h^3(v,v).
  \end{align}
  The Poincar\'e inequality~\eqref{eq:poincaregradstab} then follows directly from combining~\eqref{eq:scaled-l2-norm-gradstab}
  and~(\ref{eq:poincare-I}).
\end{proof}
\begin{remark}
  In the previous proof,
  the kick-back argument used to pass from
  $\barQ$ to the actual discrete normal projection $\Qsh$
  used only the fact that $\| \barQ - \Qsh \|_{L^{\infty}(\mcN(T))} = o(1)$
  for $h \to 0$. 
  Consequently, Proposition~\ref{prop:normalgradstab} remains
  valid when $\Gammah$ and $\{n^h_i\}_{i=1}^c$
  satisfy higher order approximation assumptions 
  of the form 
  $\| \rho \|_{L^{\infty}(\Gammah)} 
  + h  \| \Qs^e - \Qsh \|_{L^{\infty}(\Gammah)}  \lesssim h^k$
  for $k > 1$.
\end{remark}

\section{An Interpolation Operator: Construction and Estimates}
\label{sec:interpolation-properties}
The main goal of this section to construct a suitable interpolation operator
and to show that it satisfies the approximation assumption~(\ref{eq:interpolation-est-req-Ah}).
We start with the following lemma.
\begin{lemma}
  The extension operator $v^e$ defines a bounded operator
  $H^m(\Gamma) \ni v \mapsto v^e \in H^m(U_{\delta}(\Gamma))$
  satisfying the stability estimate
  \begin{align}
    \| v^e \|_{l,U_{\delta}(\Gamma)} \lesssim \delta^{c/2} \| v
    \|_{l,\Gamma}, \qquad 0 \leqslant l \leqslant m,
    \label{eq:stability-estimate-for-extension}
  \end{align}
  for $0 < \delta \leqslant \delta_0$, where the hidden constant depends only on the curvature of $\Gamma$.
\end{lemma}
\begin{proof}
  Again, by a partition of unity argument, we can assume that
  $\Gamma$ is given by a single parametrization.
  Recalling the definition of $v^e$ and
  using tube coordinates~\eqref{eq:tube-coordinates} defined by $\Phi$
  in combination with
  the measure equivalence~(\ref{eq:riemann-measure-equiv-I}),
  the tube integral for $l = 0$ computes to
  \begin{align}
    \| v^e \|_{l,U_{\delta}(\Gamma)}^2
    &=
    \int_{V}
    \left( \int_{B_{\delta}^c(0)} |u^e(y,s)|^2  \sqrt{g^\Phi(y,s)} \ds \right) \dy
    \\
    &\sim
    \int_{V}
    \left( \int_{B_{\delta}^c(0)} |u^e(y,0)|^2  \ds  \right) \sqrt{g^{\alpha}(y)} \dy
    \\
    &\sim
    \delta^{c} \int_{\Gamma} |u|^2 \dGamma.
  \end{align}
  For $l > 0$, simply combine a similar integral computation with a
  successively application of the identity $D v^e = Dv \circ Dp$
  and the boundedness of $D^l p$.
\end{proof}

Next, we recall from~\cite{ErnGuermond2006} that for
$v \in H^m(\mcT_h)$,
the standard Cl\'ement interpolant $\pi_h:L^2(N_h) \rightarrow V_h$ satisfies the
local interpolation estimates
\begin{alignat}{3}
  \| v - \pi_h v \|_{l,T}
  & \lesssim
  h^{m-l}| v |_{m,\omega(T)},
  & &\quad 0\leqslant l \leqslant m \quad &\foralls T\in \mcT_h,
  \label{eq:interpest0}
  \\
  \| v - \pi_h v \|_{l,F} &\lesssim h^{m-l-1/2}| v |_{m,\omega(F)},
  & &\quad 0\leqslant l \leqslant m - 1/2  \quad &\foralls F\in
  \mcF_{h},
  \label{eq:interpest1}
\end{alignat}
where $\omega(T)$ consists of all elements sharing a
vertex with $T$ and the patch $\omega(F)$ is defined analogously.
With the help of the extension operator,
we construct an interpolation operator via
\begin{align}
  H^m(\Gamma) \ni v \mapsto \pi_h v^e \in V_h,
  \label{eq:interpol-operator-def}
\end{align}
where we used the fact that
$
N_h = \cup_{T \in \mcT_h} T  \subset U_{\delta_0}(\Gamma)
$ for $h \lesssim \delta_0$.
Before we state and prove the main interpolation result,
we consider the interpolation error in the semi-norm $\| \cdot \|_{s_h}$
induced by the stabilization form $s_h$.
\begin{lemma}
  For $v \in H^2(\Gamma)$ and any stabilization form $s_h$ from Table~\ref{tab:cutfem-form-realizations}, it holds that
  \label{lem:interpol-est-sh}
  \begin{align}
    \| v^e - \pi_h v^e \|_{s_h}
    \lesssim
    h \| v \|_{2,\Gamma}.
  \end{align}
\end{lemma}
\begin{proof}
  For the face-based stabilization $s_h^1$, the desired estimate follows directly
  from the interpolation estimate~(\ref{eq:interpest1}), the bounded number of
  patch overlaps $\omega(F)$ and the stability result~(\ref{eq:stability-estimate-for-extension})
  \begin{align}
    \|v^e - \pi_h v^e \|^2_{s_h^1}
    &= h^{1-c} \| n_F \cdot \jump{\nabla(v^e - \pi_h v^e)} \|_{\mcF_h}^2
    \\
    & \lesssim
      h^{1-c}
      \sum_{F \in \mcF_h }h \| v^e \|_{2,\omega(F)}^2
    \lesssim
      h^{2-c} \| v^e \|_{2,U_{\delta}(\Gamma)}^2
    \lesssim
      h^{2}
      \| v \|_{2,\Gamma}^2,
  \end{align}
  where $\delta \sim h$. Similarly, we see that for the full gradient
  and normal gradient stabilization it holds
  \begin{align}
    \|v^e - \pi_h v^e \|^2_{s_h^2}
    &= h^{2-c} \| \nabla(v^e - \pi_h v^e) \|_{\mcT_h}^2
    \lesssim h^4 \| v \|_{2,\Gamma}^2,
    \\
    \|v^e - \pi_h v^e \|^2_{s_h^3}
    &= h^{\alpha-c} \| \Qsh^e\nabla(v^e - \pi_h v^e) \|_{\mcT_h}^2
    \lesssim h^{2+\alpha} \| v \|_{2,\Gamma}^2,
  \end{align}
  which in the normal gradient case gives the desired approximation order for $\alpha \geqslant 0$.
\end{proof}
To prepare the prove of the desired interpolation properties for the interpolant $\pi_h v^e$,
we recall that the standard scaled trace inequality
\begin{align}
  \| v \|_{\partial T} \lesssim h^{-\onehalf} \| v \|_{T} + h^{\onehalf} \| \nabla v \|_{T}
  \label{eq:scaled_trace_inequality}
\end{align}
is valid for $v \in H^1(T)$ and $T \in \mcT_h$.
Previous proofs~\cite{Reusken2014,BurmanHansboLarson2015} of
interpolation properties for the interpolant $\pi_h v^e$
used a similar scaled trace inequality of the form
\begin{align}
  \| v^e \|_{\Gamma_h} \lesssim h^{-\onehalf} \| v^e \|_{\mcT_h} + h^{\onehalf} \| \nabla v^e \|_{\mcT_h}
  \label{eq:scaled_trace_inequality_cut}
\end{align}
in the case where $\Gamma_h$ is a Lipschitz surface of codimension $c=1$.
We point out that the standard proof to establish such a scaled
trace inequality relies on a combination of the
divergence theorem and the fact that $\Gamma_h$ splits the element $T$ into two subdomains,
see~\cite{Grisvard1985,HansboHansboLarson2003}.
Consequently, the proof is not applicable in the case of codimension $c > 1$.
Here, we present a proof which is valid for any codimension. 
The idea is roughly to successively ``climb up'' from $\Gamma$ to the full
tubular neighborhood $U_{\delta}(\Gamma)$ via the $i$-th cubular
neighborhoods~$Q_{\delta}^i$ by using the trace inequality from
Lemma~\ref{lem:cubular-trace-estimate}.
\begin{theorem}
  \label{thm:interpolenergy}
  For $v \in H^2(\Gamma)$, it holds that
  \begin{align}
    \label{eq:interpolenergy}
    \| v^e - \pi_h v^e \|_{\Gammah}
    + h \| v^e - \pi_h v^e \|_{A_h}
    &\lesssim
    h^2 \| v \|_{2,\Gamma}.
  \end{align}
\end{theorem}
\begin{proof}
  By definition,
  $\|\cdot\|_{A_h}^2 = \| \cdot \|_{a_h}^2 + \|\cdot \|_{s_h}^2$
  and thanks to Lemma~\ref{lem:interpol-est-sh}
  and the choices of $a_h$ given in Section~\ref{sec:cutfem-realizations},
  it holds to prove that
  \begin{align}
    \| v^e - \pi_h v^e \|_{\Gammah}
    +
    h \| D(v^e - \pi_h v^e) \|_{\Gamma_h}
    \lesssim h^2 \| v \|_{2,\Gamma}.
  \end{align}
  Clearly, $\Gamma$ can be covered by local coordinate neighborhoods satisfying the assumptions
  of Lemma~\ref{lem:cubular-trace-estimate}.
  The quasi-uniformness of $\mcT_h$
  and the fact that $\dist(\Gamma, T) \lesssim h$ implies
  that there is a similar simplex $\widetilde{T}(T)$ of $\diam(\widetilde{T}) \sim h$
  such that the chain of inclusions
  $T \subset Q_{\delta}^c(p(T)) \subset \widetilde{T}$
  holds with $\delta \sim h$.
  Consequently, there is a $\widetilde{v}_T \in P_1(\widetilde{T})$
  satisfying the interpolation estimate
  \begin{align}
    \| v^e - \widetilde{v}_T \|_{\widetilde{T}}
    + h \| D(v^e - \widetilde{v}_T) \|_{\widetilde{T}}
    \lesssim
    h^2 \| v^e \|_{H^2(\widetilde{T})}.
    \label{eq:interpol-est-tildeT}
  \end{align}
  Restricting $\widetilde{v}_T$ to $\Gamma$ and denoting its
  subsequent extension $(\widetilde{v}_T|_{\Gamma})^e$ simply by $\widetilde{v}_T^e$,
  we obtain
  \begin{align}
    \nonumber
    \| v^e - \pi_h v^e \|_{\Gamma_h}^2 + h^2 \| D(v^e - \pi_h v^e) \|_{\Gamma_h}^2
    &\lesssim
    \sum_{T \in \mcT_h}
    \Bigl(
    \| v^e - \widetilde{v}_T^e \|_{T \cap \Gamma_h}^2
    + h^2 \| D(v^e - \widetilde{v}_T^e) \|_{T \cap \Gamma_h}^2
    \Bigr)
    \\
    &\quad +
    \sum_{T \in \mcT_h}
    \Bigl(
    \| \widetilde{v}_T^e - \pi_h v^e  \|_{T \cap \Gamma_h}^2
    + h^2 \| D(\widetilde{v}_T^e - v^e) \|_{T \cap \Gamma_h}^2
    \Bigr)
    \\
    &= I + II,
  \end{align}
  which we estimate next.
  \\
  {\bf Term $\bfI$.} We start with lifting each discrete manifold part $\Gamma_h \cap T$
  to $\Gamma$ which gives
  \begin{align}
    I
    &\lesssim
    \sum_{T\in \mcT_h}
    \| v - \widetilde{v}_T \|_{(T \cap \Gamma_h)^l}^2
    + h^2 \| D(v - \widetilde{v}_T) \|_{(T \cap \Gamma_h)^l}^2
    \\
    &\lesssim
    \sum_{T\in \mcT_h}
    \| v - \widetilde{v}_T \|_{p(T)}^2
    + h^2 \| D(v - \widetilde{v}_T) \|_{p(T)}^2.
    \label{eq:interpol-projected-elements}
  \end{align}
Then apply the scaled trace estimate~(\ref{eq:cubular-trace-estimate}) on each projected element
$Q_{\delta}^0(p(T)) = p(T) \subset \Gamma$
to see that
\begin{align}
  \| v - \widetilde{v}_T \|_{Q_{\delta}^0(p(T))}^2
  + h^2 \| D(v - \widetilde{v}_T) \|_{Q_{\delta}^0(p(T))}^2
  &\lesssim
    h^{-1}
  \Bigl(
  \| v^e - \widetilde{v}_T \|_{Q_{\delta}^1(p(T))}^2
  + h^2 \| D(v^e - \widetilde{v}_T) \|_{Q_{\delta}^1(p(T))}^2
  \nonumber
  \\
  &\quad + h^4 \| D^2v^e \|_{Q_{\delta}^1(p(T))}
  \Bigr).
\end{align}
After reiterating the argument and applying~(\ref{eq:cubular-trace-estimate}) to
$\| \cdot \|_{Q^i_h(p(T))}$ for $i = 1,\ldots, c$, we arrive at
\begin{align}
  \| v - \widetilde{v}_T \|_{Q_{\delta}^0(p(T))}^2
  + h^2 \| D(v - \widetilde{v}_T) \|_{Q_{\delta}^0(p(T))}^2
    &\lesssim
    h^{-c}
    \Bigl(
    \| v^e - \widetilde{v}_T \|_{Q_{\delta}^{c}(p(T))}^2
    + h^2\| D(v^e - \widetilde{v}_h) \|_{Q_{\delta}^c(p(T))}^2
    \Bigr)
    \nonumber
    \\
    &\quad
    + \sum_{i=1}^c
    h^{4-i} \| D v^e \|_{Q_{\delta}^{i}(p(T))}^2
    \\
    &= I_{a} + I_b.
\end{align}
Recalling that $Q_{\delta}^c(p(T)) \subset \widetilde{T}$ and that
$\widetilde{v}_T$ satisfies~(\ref{eq:interpol-est-tildeT}),
the term $I_a$ can be further estimated,
\begin{align}
  \sum_{T\in \mcT_h} I_a
  \lesssim h^{-c} h^4 \| D^2 v^e \|_{\widetilde{T}(T)}^2
  \lesssim h^{-c} h^4 \| D^2 v^e \|_{Q_{\delta}(\Gamma)}^2
  \lesssim h^4 \| v \|_{2,\Gamma}^2,
\end{align}
where we used the stability estimate~(\ref{eq:stability-estimate-for-extension})
and the fact that the number
$\#\{ T' \in \mcT_h : \widetilde{T}(T) \cap \widetilde{T}(T') \neq \emptyset\}$
is uniformly bounded in $T$.
Similarly, each projected element $p(T)$ is only overlapped by a uniformly bounded
number of other projected elements $p(T')$ and therefore
the remaining term $I_b$ can be bounded by
\begin{align}
  \sum_{T \in \mcT_h} I_b
  &\lesssim
  \sum_{i=1}^c \sum_{T \in \mcT_h}
  h^{4} \| D^2 v \|_{p(T)}
  \lesssim h^{4} \| D^2 v \|_{\Gamma},
\end{align}
where in the first step, a stability estimate of the form~(\ref{eq:stability-estimate-for-extension}) with $U_{\delta}(\Gamma)$ replaced by $Q_{\delta}^i(\Gamma)$ was used
for $i = 1,\ldots,c$ and $\delta \sim h$.
\\
{ \bf Term II.}
A successive application of
the inverse inequalities~(\ref{eq:inverse-estimate-cut-v-on-K}),
(\ref{eq:inverse-estimates-standard})
and a triangle inequality yields to
\begin{align}
  II
     &\lesssim
    h^{-c} \| \widetilde{v}_T^e - \pi_h v^ e \|_{\mcT_h}^2
  \lesssim
    h^{-c}\| v^e - \pi_h v^ e \|_{\mcT_h}^2
    + h^{-c}\| v^e - \widetilde{v}_T^e \|_{\mcT_h}^2
  = II_a +  II_b.
\end{align}
With the interpolation estimate~(\ref{eq:interpest0}) and
stability bound~(\ref{eq:stability-estimate-for-extension}),
term $II_a$ can be estimated by
\begin{align}
  II_a \lesssim h^{-c} h^4 \| D^2 v^e \|_{\mcT_h}^2
  \lesssim h^{4} \| v \|_{2,\Gamma}^2.
\end{align}
Referring to (\ref{eq:interpol-projected-elements}),
the remaining term $II_b$ can be treated exactly as Term $I$ by observing that
\begin{align}
  II_b
  \lesssim
  \sum_{T \in \mcT_h}
  h^{-c}\| v^e - \widetilde{v}_T^e \|_{U^c_h(p(T))}^2
  \lesssim
  \sum_{T \in \mcT}
     \| v^e - \widetilde{v}_T^e \|_{p(T)}^2
  \lesssim h^4 \| v \|_{2,\Gamma}^2.
\end{align}
which concludes the proof.
\end{proof}

\section{Verification of the Quadrature and Consistency Error Estimates}
\label{sec:quadrature-consistency-error}
Finally, with the help of the geometric estimate established in
Section~\ref{sec:geometric-estimates},
we now show that for the proposed cut finite element realizations
the quadrature and consistency error satisfy
assumption~(\ref{eq:quadrature-l-est-primal-req})--(\ref{eq:quadrature-a-est-dual-req})
and~(\ref{eq:consist-err-est-req}).
\begin{lemma}
  \label{lem:consistency-error-est}
  Let the discrete linear form $l_h$ be defined by ~\eqref{eq:lh-def} and
  assume that the discrete bilinear $a_h$ is given by either $a_h^1$ or $a_h^2$
  from Table~\ref{tab:cutfem-form-realizations}.
  Then
  \begin{align}
    | l_h(v) - l(v^l) | &\lesssim h^2 \| f \|_{\Gamma} \|v \|_{A_h} \quad\foralls v \in V_h,
                          \label{eq:quadrature-l-est-primal}
    \\
    | a(u,v^l) - a_h(u^e,v) | &\lesssim h \| u \|_{2,\Gamma} \| v \|_{A_h}
    \quad \foralls u \in H^2(\Gamma), \foralls v \in V_h.
                                \label{eq:quadrature-a-est-primal}
  \end{align}
  Furthermore, for $\phi \in H^2(\Gamma)$ and $\phi_h = \pi_h \phi$
  the following improved estimates hold
  \begin{align}
    \label{eq:quadrature-l-est-dual}
    | l_h(\phi_h) - l(\phi_h^l) | &\lesssim h^2 \| f \|_{\Gamma} \|\phi\|_{2,\Gamma},
    \\
    \label{eq:quadrature-a-est-dual}
    | a(u_h^l,\phi_h^l) - a_h(u_h,\phi_h) | &\lesssim h^2 \| u \|_{2,\Gamma} \| \phi \|_{2,\Gamma}.
  \end{align}
\end{lemma}
\begin{proof}
  We start with proving \eqref{eq:quadrature-l-est-primal}.
  For the quadrature error of $l_h$ side we have
  \begin{align}
    l(v^l)- l_h(v) &= (f,v^l)_\Gamma - (f^e, v)_\Gammah
    = (f,v^l (1 -|B|_d^{-1}))_\Gamma
    \lesssim h^2  \| f \|_{\Gamma}\|v^l\|_{\Gamma}
    \lesssim h^2  \| f \|_{\Gamma}\| v \|_{A_h},
  \end{align}
  where in the last step, the Poincar\'e
  inequality~\eqref{eq:poincare-I} was used
  after passing from $\Gamma$ to $\Gamma_h$.
  Since the interpolation estimate~\eqref{eq:interpolation-est-req-Ah}
  yields the simple bound $\| \pi_h \phi \|_{A_h} \lesssim \| \phi \|_{2,\Gamma}$,
  estimate~\eqref{eq:quadrature-l-est-dual} follows immediately.

  Turning to estimate~\eqref{eq:quadrature-a-est-primal} and \eqref{eq:quadrature-a-est-dual}
  and applying the splitting $ \nabla = \nablash + \Qsh \nabla$ we see that
  \begin{align}
    a_h^2(u^e, v)
    = (\nablash u^e, \nablash v)_{\Gamma_h}
    + (\Qsh \nabla u^e, \Qsh \nabla v)_{\Gamma_h}
    = a_h^1(u^e, v)
    + (\Qsh \nabla u^e, \Qsh \nabla v)_{\Gamma_h},
  \end{align}
  and thus it is enough to consider only the case $a_h = a_h^{2}$.
  Using this decomposition we obtain
  \begin{align}
    a(u, v^l) - a_h(u^e, v)
    &=
    \bigl(
    (\nablas u, \nablas v^l)_{\Gamma}
    - (\nablash u^e, \nablash v)_{\Gamma_h}
    \bigr)
    - (\Qsh \nabla u^e, \Qsh \nabla v)_{\Gamma_h}
    \\
    &= I  + II
    \label{eq:ah-split}
  \end{align}
  A bound for the first term $I$ can be derived by
  lifting the tangential part of $a_h(\cdot,\cdot)$ to $\Gamma$
  and employing the bounds for
  determinant~\eqref{eq:detBbound},
  the operator norm estimates~\eqref{eq:BBTbound},
  and the norm
  equivalences~\eqref{eq:norm-equivalences-ve}--\eqref{eq:norm-equivalences-wh},
  \begin{align}
    I
    &=
    (\nablas u,\nablas v^l)_{\Gamma} - (\nablash u,\nablash
    v)_{\Gammah}
    \\
    &=
    (\nablas u,\nablas v^l)_{\Gamma} - ((\nablash u)^l,(\nablash
    v)^l |B|^{-1})_{\Gamma}
    \\
    &=
    \left(
    (\Ps - |B|^{-1} B B^T)\nablas u,\nablas v^l
    \right)_{\Gamma}
    \\
    &=
      \left(
      (\Ps - B B^T)+ (1-|B|^{-1}) B B^T)\nablas u,\nablas v^l
      \right)_{\Gamma}
    \\
    &\lesssim h^2 \| f \|_{\Gamma} \|\nablas v^l\|_{\Gamma}.
  \end{align}
  Again, for $\phi_h = \pi_h \phi$,
  the improved estimate ~\eqref{eq:quadrature-a-est-dual} follows
  from
  $
  \| \nablas \phi_h^l\|_{\Gamma}
  \lesssim \|\phi_h \|_{A_h}
  \lesssim \|\phi \|_{2,\Gamma}
  $.
  Turning to the second term $I$ and applying
  the inequality~\eqref{eq:normal-grad-est} to $\Qsh \nabla u^e$
  gives
  \begin{align}
    II_b
    & \lesssim
    \|\Qsh \nabla u^e \|_{\Gammah}
    \|\Qsh \nabla v \|_{\Gammah}
    \\
    &
    \lesssim
    h \| f \|_{\Gamma}
    \|\Qsh \nabla v \|_{\Gammah}
    \label{eq:estimate-IIb}
  \end{align}
  For general $v \in V_{h}$, the last factor in $II_b$
  is simply bounded by $\| \nabla v \|_{\Gamma_h}$
  while in the special case
  $v = \pi_h \phi^e$,
  the interpolation estimate~\eqref{eq:interpolenergy}
  and
  a second application of~\eqref{eq:normal-grad-est} to $\Qsh \nabla
  \phi^e$ yields
  \begin{align}
    \|\Qsh \nabla \pi_h \phi^e \|_{\Gammah}
    \lesssim
    \|\Qsh \nabla \phi^e \|_{\Gammah}
    +
    \|\Qsh \nabla (\phi^e - \pi_h \phi^e)\|_{\Gammah}
    \lesssim
    h \| \phi \|_{2,\Gamma}
  \end{align}
\end{proof}
We conclude this section by commenting on the consistency of the proposed cut finite element
formulations. First note that $s_h^1(u^e, u^e) = 0$
for $u \in H^2(\Gamma)$
since $u^e \in H^2(U_{\delta}(\Gamma))$.
On the other hand,
the stability  estimate~\eqref{eq:stability-estimate-for-extension}
with $\delta \sim h$ shows that
$
  h^{2-c} \|\nabla u^e\|_{\mcT_h}^2
  \lesssim
  h^2 \| u \|_{2,\Gamma}^2
$
and thus the weak consistency assumption~(\ref{eq:consist-err-est-req}) holds for $s_h^2$.
Finally, for the normal gradient $s_h^3$ we have
\begin{align}
  h^{\alpha - c}
  \| \Qsh \nabla u^e \|_{\mcT_h}^2 =
  h^{\alpha - c}
  \| (\Qsh - \Qs) \nabla u^e \|_{\mcT_h}^2
  \lesssim
  h^{\alpha -c + 2} \| \nabla u^e \|_{\mcT_h}^2
  \lesssim
  h^{\alpha + 2} \| \nablas u \|_{\Gamma}^2
\end{align}
and thus any choice $\alpha \in [0,2]$ ensures a weakly consistent stabilization
which satisfies the Poincar\'e inequality~(\ref{eq:discrete-poincare-Ah-abstract}).

\section{Numerical Results}
\label{sec:numerical_results}
In this final section we complement the
development of the theoretical framework
with a number of numerical studies which validate the theoretically proven bounds
on condition number and the a priori error
as stated in Theorem~\ref{thm:aprioriest} and
\ref{thm:condition-number-estimate}, respectively.
With $\RR^3$ as embedding space,
we consider examples for codimension $1$ and $2$.

\begin{figure}[htb]
  \begin{center}
    \begin{subfigure}[c]{0.48\textwidth}
      \includegraphics[width=1.00\textwidth]{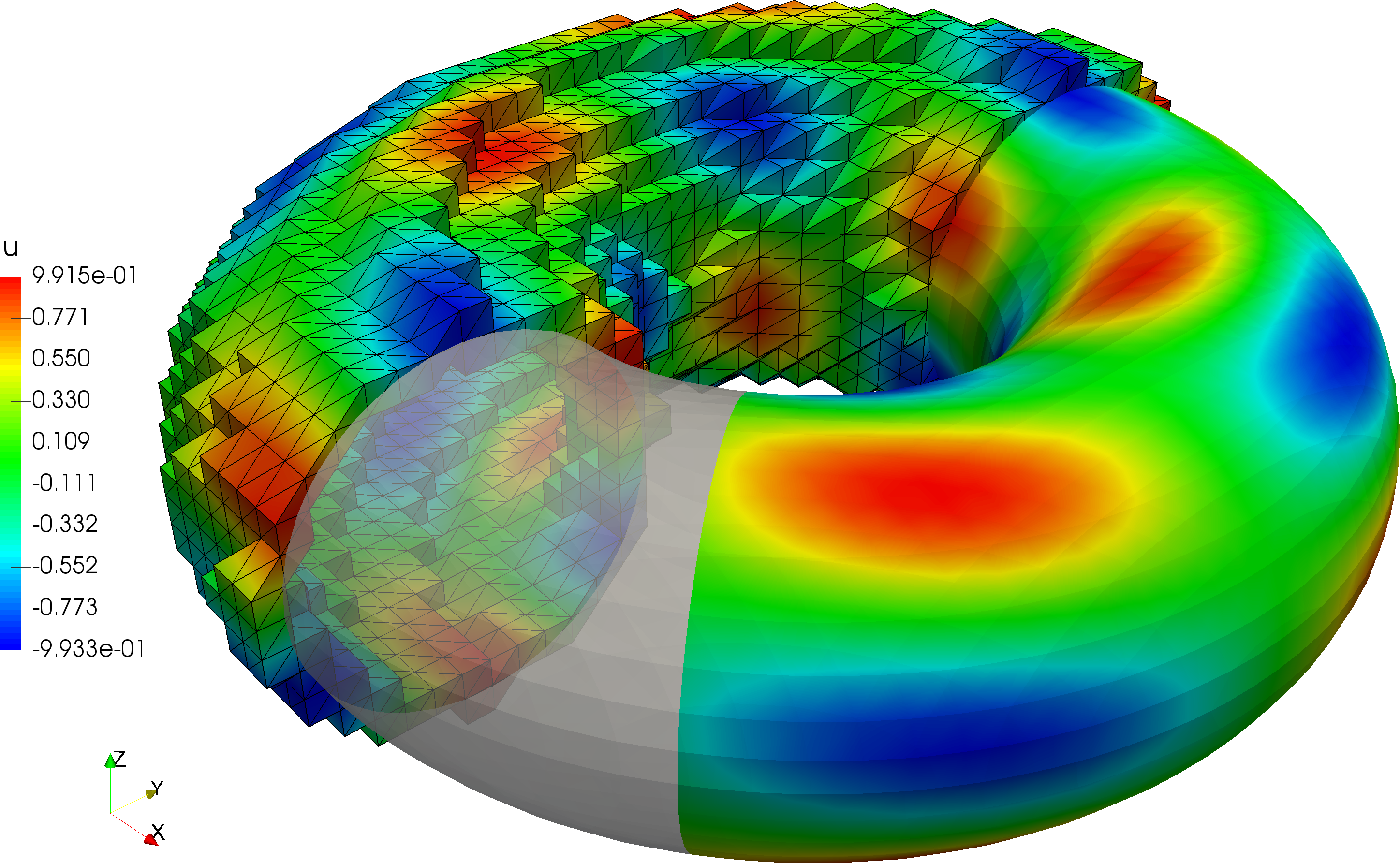}
    \end{subfigure}
    \begin{subfigure}[c]{0.48\textwidth}
      \includegraphics[width=1.0\textwidth]{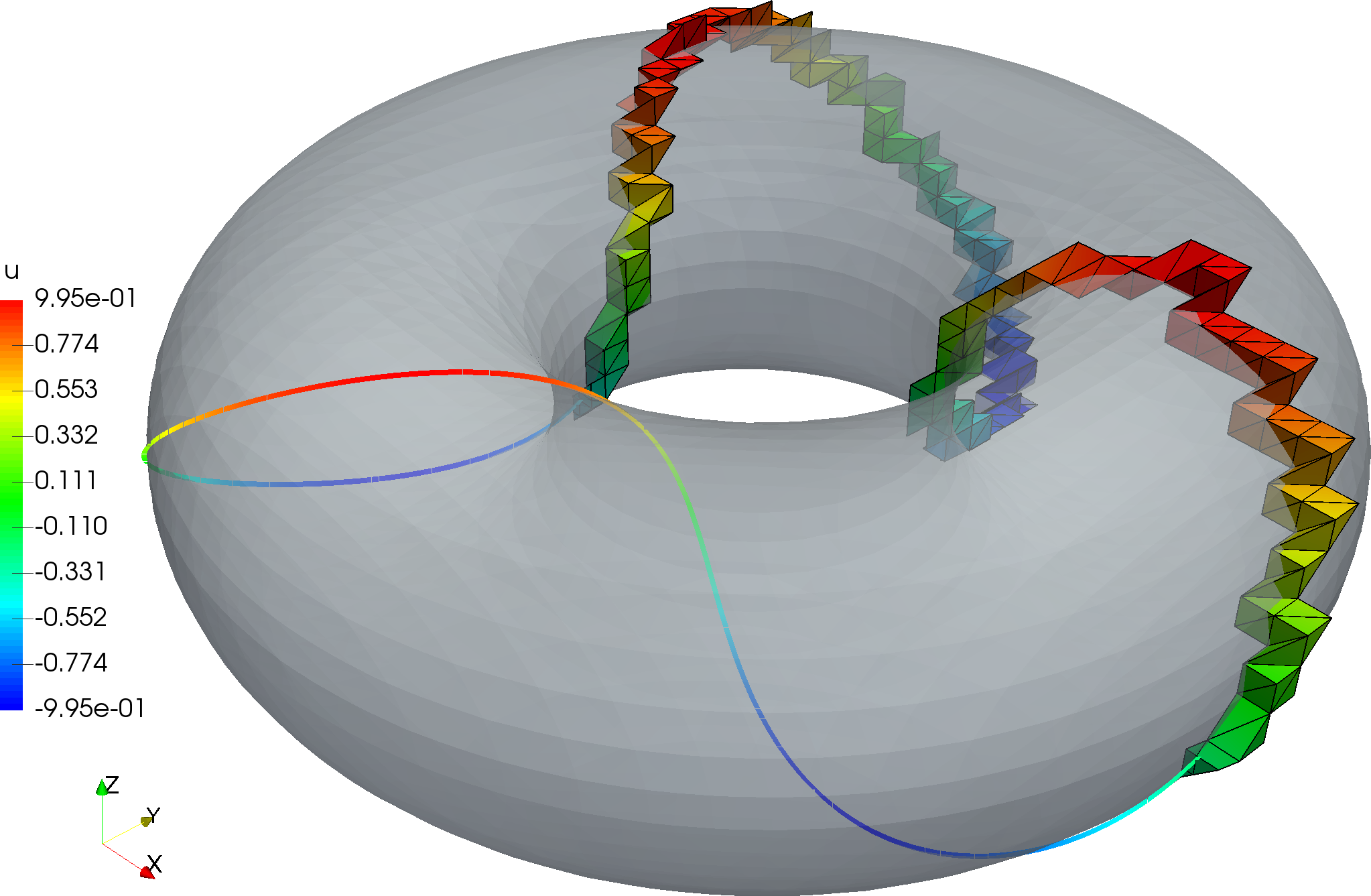}
    \end{subfigure}
  \end{center}
  \caption{Solution plots. Each plot shows the numerical solution $u_h$
      computed on the active mesh $\mcT_h$ and its restriction
      to the manifold discretization $\mcK_h$.
      (Left) Solution for the surface example
      computed on $\mcT_2$ with $h \approx 0.22\cdot10^{-2}$
      using the normal gradient stabilized tangential form $a^1_h + \tau
      s^3_h$ with $\tau = 1.0$.
      (Right) Solution for curve example on $\mcT_2$ with same mesh
      size using the full gradient stabilized
    full gradient form $a^2_h + \tau s^2_h$ with $\tau = 1.0$.}
  \label{fig:solution-plot}
\end{figure}

\subsection{Convergence Rates for the Laplace-Beltrami Problem on a Torus}
\label{ssec:convergence-rates-tests}
In the first convergence rate study, we define total bilinear form~$A_h$ by
combining the full gradient form $a_h^2$
with the normal gradient stabilization $s_h^3$,
\begin{align}
  A_h(u_h, v_h)
  &= (\nabla u_h, \nabla v_h)_{\mcK_h}
  + \tau h (n_{\Gammah} \cdot \nabla u_h, n_{\Gammah} \cdot \nabla v_h)_{\mcT_h}
\end{align}
with $\tau = 0.1$
to discretize the
Laplace-Beltrami type problem
\begin{align}
  -\Delta_{\Gamma} u + u = f \quad \text{on } \Gamma
  \label{eq:laplace-beltrami-type-problem}
\end{align}
on the torus surface $\Gamma$ defined by
\begin{align}
  \Gamma
  =
  \{
    x \in \RR^3 : r^2 = x_3^2 + (\sqrt{x_1^2 + x_2^2} -R)^2
  \}
\label{eq:torus-levelset}
\end{align}
with major radius $R = 1.0$ and minor radius $r = 0.5$. Based on the parametrization
\begin{align}
  x =
  \gamma(\phi, \theta) =
  R
  \begin{pmatrix}
    \cos \phi \\
    \sin \phi \\
    0 \\
  \end{pmatrix}
  + r
  \begin{pmatrix}
    \cos \phi \cos \theta \\
    \sin \phi \cos \theta  \\
    \sin \theta \\
  \end{pmatrix},
  \quad (\phi,\theta) \in [0,2\pi) \times [0,2\pi),
  \label{eq:torus-parametrization}
\end{align}
an analytical reference solution $u$ with corresponding right-hand side
$f$ is given by
\begin{align}
  u(x) &= \sin(3 \phi)\cos(3\theta + \phi),
  \\
  f(x) &= r^{-2} (9 \sin(3\phi)\cos(3\theta + \phi)
  \nonumber
  \\
  &\quad + (R + r \cos(\theta))^{-2}
  (10 \sin(3\phi)\cos(3\theta + \phi) + 6\cos(3\phi)\sin(3\theta +
  \phi))
  \nonumber
  \\
  &\quad + (r(R + r \cos(\theta)))^{-1}
  (3 \sin(\theta)\sin(3\phi)\sin(3\theta + \phi)
  )
  + u(x(\phi,\theta)).
\end{align}
To examine the convergence rates predicted by Theorem~\ref{thm:aprioriest},
we generate a sequence of meshes $\{\mcT_k\}_{k=0}^5$
by uniformly refining an initial structured background mesh $\widetilde{\mcT}_0$
for $\Omega = [-1.1,1.1]^3 \supset \Gamma$ with mesh size $h=0.22$.
At each refinement
level $k$, the mesh $\mcT_k$ is then given by the active (background) mesh
as defined in \eqref{eq:narrow-band-mesh}.
For a given error norm, the corresponding
experimental order of convergence (EOC) at
refinement level $k$ is calculated using the formula
\begin{align*}
  \text{EOC}(k) = \dfrac{\log(E_{k-1}/E_{k})}{\log(2)}
\end{align*}
with $E_k$ denoting the error of the computed solution $u_k$ at
refinement level $k$.
The resulting errors for the
$\| \cdot \|_{H^1(\Gamma_h)}$
and $\|\cdot\|_{L^2(\Gamma_h)}$ norms
are summarized in
Table~\ref{tab:convergence-rates-example} (left)
and
confirm the first-order and second-order
convergences rates predicted by Theorem~\ref{thm:aprioriest}.
Finally, the discrete solution computed at refinement level $k=2$
is visualized in Figure~\ref{fig:solution-plot} (left).

\subsection{Convergence Rates for the Laplace-Beltrami Problem on a Torus Line}
Next, we solve problem~\eqref{eq:laplace-beltrami-type-problem}
on the $1$-dimensional manifold $\Gamma \subset \RR^3$ defined by
the torus line
\begin{align}
  x =
  \gamma(t, N t) =
  R
  \begin{pmatrix}
    \cos (t) \\
    \sin (t) \\
    0 \\
  \end{pmatrix}
  + r
  \begin{pmatrix}
    \cos (t) \cos(N t)\\
    \sin (t) \cos(N t) \\
    \sin(N t) \\
  \end{pmatrix},
  \quad t \in [0,2\pi),
  \label{eq:torus-lines}
\end{align}
where $N$ determines ``the winding number'' of the curve $\gamma$ with respect to the
torus centerline $\{x \in \RR^3: x_1^2 + x_2^2 = R^2 \wedge x_3 = 0 \}$.
We set $R = 2 r = 1.0$ and $N = 3$.
This time,
the full gradient form
$a_h^2$
augmented by the full gradient stabilization
$s_h^2$
constitutes the overall bilinear form
\begin{align}
  A_h(u_h, v_h)
  & = (\nabla u_h, \nabla v_h)_{\mcK_h} + \tau h (\nabla u_h, \nabla v_h)_{\mcT_h}
\end{align}
to compute the discrete solution $u_h$.
To examine the convergence properties of the scheme,
we consider the manufactured solution given by
\begin{align}
  u(x) &= \sin(3t)
  \\
  f(x) &=
    36\frac{- 64 \sin^{5}{\left (t \right )} -
      128 \sin^{4}{\left (t \right )} \sin{\left (3 t \right )} + 2
      \sin{\left (3 t \right )} \cos{\left (3 t \right )} + 41
      \sin{\left (3 t \right )} - 28 \sin{\left (5 t \right )} + 8
    \sin{\left (7 t \right )}}
    {\left(128 \sin^{6}{\left (t \right )} -
        192 \sin^{4}{\left (t \right )} + 72 \sin^{2}{\left (t \right
        )} - 9 \sin^{2}{\left (3 t \right )} + 4 \cos{\left (3 t
  \right )} + 14\right)^{2}}
  \label{eq:convergence-example-2-f}
\end{align}
Similar to the previous example, we generate a series of successively refined
active background meshes $\{\mcT_k\}_{k=0}^5$ with
$h_k = 0.22/N_k$ and $N_k = 10\cdot 2^k$.
To  define a suitable discretization of $\Gamma$,
we first subdivide the parameter interval $[0,2\pi)$ into
$10\cdot N_k$ subintervals of equal length.
The collection of segments connecting the
mapped endpoints of each subinterval to $\Gamma$
defines an initial partition $\widetilde{\mcK}_k$ of the curve
$\gamma$. Then a compatible partition $\mcK_k$ is
generated by computing all non-trivial intersections $K\cap T$
for $K \in \widetilde{\mcK}_k$, $T \in \mcT_k$
and partitioning each segment $K$
accordingly.
A plot of the computed solution at refinement level $k=2$ is shown
in Figure~\ref{fig:solution-plot} (right).
As before, the observed reduction of the
$\| \cdot \|_{H^1(\Gamma_h)}$ and $\|\cdot\|_{L^2(\Gamma_h)}$
discretization error
confirms the predicted convergences rates, see
Table~\ref{tab:convergence-rates-example} (right).
\begin{table}[htb]
  \footnotesize
  \centering
  \begin{center}
    \input{tables/results_example_torus_full_grad_normal_grad_stab_tau_0.1-error-table.tex}
    \hspace{2ex}
    \input{tables/results_example_torus_line_tau_1.0-error-table.tex}
  \end{center}
  \vspace{1ex}
  \caption{Convergence rates for the surface example (left)
  and curve example (right).}
  \label{tab:convergence-rates-example}
\end{table}

\subsection{Condition Number Tests}
\label{ssec:condition-number-tests}
The final section is devoted to the numerical study of the dependency
of the condition number on the mesh size and on the positioning of the
embedded manifold in the background mesh.
Again, we consider the case of a surface and a curve embedded into $\RR^3$
and pick the unit-sphere $S^2 = \{ x \in \RR^3 : \| x \| = 1 \}$
and the torus line defined by~\eqref{eq:torus-lines}
as example manifolds of codimension $1$ and $2$, respectively.
For each case, we choose the same bilinear form $A_h$ as in
the corresponding convergence rate test.

For each considered manifold $\Gamma$,
we generate a sequence
$\{\mcT_k\}_{k=0}^5$ of tessellations of $\Omega = [-1.6, 1.6]^3$ with
mesh size $h = 3.2/k$ for $k \in \{10, 15, 20, 30, 40, 60\}$.
To study the influence of the relative position on the condition number,
we generate for each mesh $\mcT_k$ a family of manifolds
$\{\Gamma_{\delta}\}_{0\leqslant\delta\leqslant 1}$
by translating $\Gamma$
along the diagonal $(h,h,h)$; that is,
$\Gamma_{\delta} = \Gamma + \delta (h,h,h)$ with $\delta \in [0,1]$.
For the surface example,
we compute the condition number $\kappa_{\delta}(\mcA)$
for $\delta = l/500$, $l=0,\ldots,500$,
as the ratio
of the absolute value of the largest (in modulus) and smallest (in
modulus), non-zero eigenvalue.
For the curve example, a higher sampling rate defined by
$\delta = l/10000$, $l=0,\ldots,10000$
is used to reveal the high number of strong peaks in the condition
number plots for the unstabilized method.
To study the $h$ dependency of the condition number,
the minimum, maximum, and the arithmetic mean
of the scaled condition numbers $h^{2}\kappa_{\delta}(A)$
are computed for each mesh size $h$.
The resulting numbers displayed in Table~\ref{tab:scaled-condition-number}
confirm the $O(h^{-2})$ bound proven in Theorem~\ref{thm:condition-number-estimate}.
\begin{table}[htb]
  \begin{subtable}[t]{1.0\textwidth}
    \centering
    \input{tables/condition_number_rates_table_2d_3d.tex}
    \\[1ex]
    \caption{Translated surface example computed with bilinear form $a_h^2(v,w) + \tau s_h^3(v,w)$}
  \end{subtable}
  \begin{subtable}[t]{1.0\textwidth}
    \centering
    \input{tables/condition_number_rates_table_1d_3d.tex}
    \\[1ex]
    \caption{Translated curve example computed with bilinear form $a_h^2(v,w) + \tau s_h^2(v,w)$}
  \end{subtable}
  \caption{Minimum, maximum, and arithmetic mean of the scaled condition
  number for various mesh sizes $h$.}
  \label{tab:scaled-condition-number}
\end{table}

\begin{figure}[htb!]
  \begin{subfigure}[t]{0.49\textwidth}
    \includegraphics[page=1,width=1\textwidth]{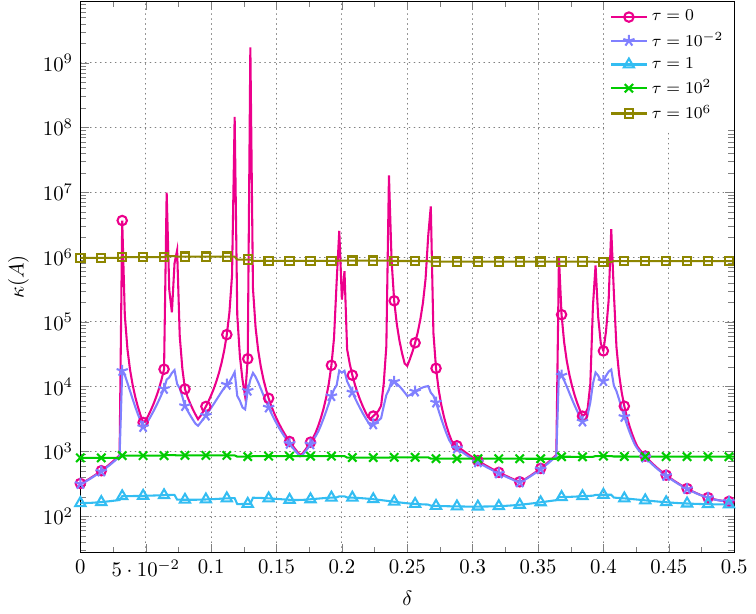}
  \end{subfigure}
    \begin{subfigure}[t]{0.49\textwidth}
      \includegraphics[page=1,width=1\textwidth]{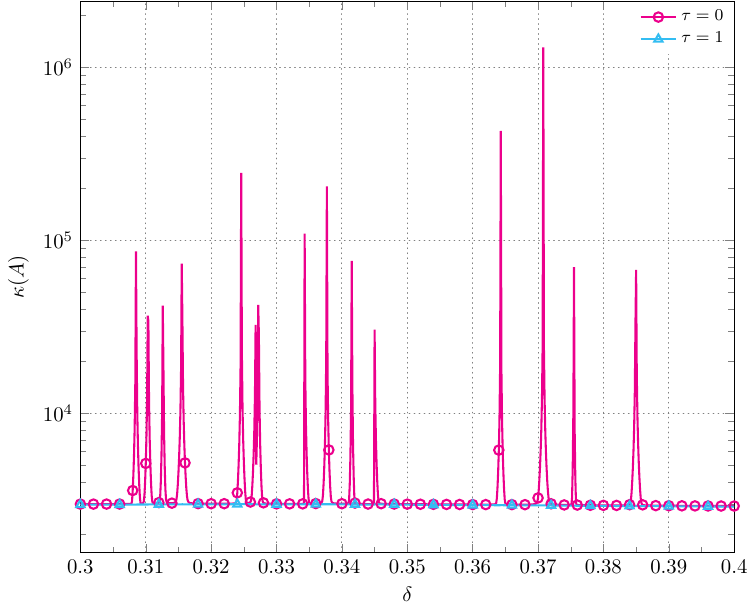}
    \end{subfigure}
    \caption{Condition numbers plotted as a function of the position parameter $\delta$
      for different stabilizations and penalty parameters.
      (Left) Surface example where a combination of the full gradient form $a_h^2(v,w)$
      and the normal gradient stabilization $s_h^3(v,w)$ was used.
      (Right) Curve example, with a combination of the full gradient form $a_h^2(v,w)$
      and the full gradient stabilization $s_h^2(v,w)$.
      Note the different $x$ axis range for the surface and curve example.}
    \label{fig:condition_number-example}
\end{figure}
In Figure~\ref{fig:condition_number-example}, the condition numbers
computed on $\mcT_2$ are plotted as a function of the position parameter delta.
For the surface example, different stabilization parameters $\tau$ for
the normal-gradient stabilization are tested and the resulting plots
show clearly that the computed condition numbers are robust with respect
to the translation parameter $\delta$ when $\tau$ is chosen large enough, i.e.
$\tau \sim 1$.
In contrast, the condition number is highly sensitive
and exhibits high peaks as a function of $\delta$ if we set the penalty
parameter $\tau$ to $0$.
Note that for very large parameters, the size of the condition numbers, while
robust with respect to $\delta$, increases again.
For the curve example, we observe a similiar, albeit more extreme behavior
as the condition number distribution shows more frequent and stronger peaks
in the unstabilized case. Here, the full gradient stabilization $s_h^2(v,w)$ was
employed. An additional $\tau$ parameter study
gave very similar results to the studies performed
in~\cite{BurmanHansboLarsonEtAl2016c} for full gradient stabilized surface PDE methods,
and thus is not included here.
In particular, we observed that the condition numbers do not increase again for
excessively high choices of $\tau$.
\section*{Acknowledgements}
This research was supported in part by EPSRC, UK, Grant
No. EP/J002313/1, the Swedish Foundation for Strategic Research Grant
No.\ AM13-0029, the Swedish Research Council Grants Nos.\ 2011-4992,
2013-4708, and Swedish strategic research programme eSSENCE.

\bibliographystyle{plainnat}
\bibliography{bibliography}

\end{document}

%% file: tables/results_example_torus_full_grad_normal_grad_stab_tau_0.1-error-table.tex
\begin {tabular}{cr<{\pgfplotstableresetcolortbloverhangright }@{}l<{\pgfplotstableresetcolortbloverhangleft }cr<{\pgfplotstableresetcolortbloverhangright }@{}l<{\pgfplotstableresetcolortbloverhangleft }c}%
\toprule $k$&\multicolumn {2}{c}{$\|u_k - u \|_{1,\Gamma _h}$}&EOC&\multicolumn {2}{c}{$\|u_k - u \|_{\Gamma _h}$}&EOC\\\midrule %
\pgfutilensuremath {0}&$9.99$&$\cdot 10^{0}$&--&$1.16$&$\cdot 10^{0}$&--\\%
\pgfutilensuremath {1}&$5.54$&$\cdot 10^{0}$&\pgfutilensuremath {0.85}&$4.33$&$\cdot 10^{-1}$&\pgfutilensuremath {1.43}\\%
\pgfutilensuremath {2}&$2.80$&$\cdot 10^{0}$&\pgfutilensuremath {0.98}&$1.18$&$\cdot 10^{-1}$&\pgfutilensuremath {1.87}\\%
\pgfutilensuremath {3}&$1.42$&$\cdot 10^{0}$&\pgfutilensuremath {0.98}&$3.05$&$\cdot 10^{-2}$&\pgfutilensuremath {1.95}\\%
\pgfutilensuremath {4}&$7.14$&$\cdot 10^{-1}$&\pgfutilensuremath {0.99}&$7.74$&$\cdot 10^{-3}$&\pgfutilensuremath {1.98}\\%
\pgfutilensuremath {5}&$3.58$&$\cdot 10^{-1}$&\pgfutilensuremath {1.00}&$1.95$&$\cdot 10^{-3}$&\pgfutilensuremath {1.99}\\\bottomrule %
\end {tabular}%

%% file: tables/results_example_torus_line_tau_1.0-error-table.tex
\begin {tabular}{cr<{\pgfplotstableresetcolortbloverhangright }@{}l<{\pgfplotstableresetcolortbloverhangleft }cr<{\pgfplotstableresetcolortbloverhangright }@{}l<{\pgfplotstableresetcolortbloverhangleft }c}%
\toprule $k$&\multicolumn {2}{c}{$\|u_k - u \|_{1,\Gamma _h}$}&EOC&\multicolumn {2}{c}{$\|u_k - u \|_{\Gamma _h}$}&EOC\\\midrule %
\pgfutilensuremath {0}&$1.77$&$\cdot 10^{0}$&--&$8.59$&$\cdot 10^{-1}$&--\\%
\pgfutilensuremath {1}&$7.48$&$\cdot 10^{-1}$&\pgfutilensuremath {1.24}&$2.74$&$\cdot 10^{-1}$&\pgfutilensuremath {1.65}\\%
\pgfutilensuremath {2}&$3.75$&$\cdot 10^{-1}$&\pgfutilensuremath {1.00}&$6.66$&$\cdot 10^{-2}$&\pgfutilensuremath {2.04}\\%
\pgfutilensuremath {3}&$1.91$&$\cdot 10^{-1}$&\pgfutilensuremath {0.98}&$1.71$&$\cdot 10^{-2}$&\pgfutilensuremath {1.96}\\%
\pgfutilensuremath {4}&$9.77$&$\cdot 10^{-2}$&\pgfutilensuremath {0.96}&$4.36$&$\cdot 10^{-3}$&\pgfutilensuremath {1.97}\\%
\pgfutilensuremath {5}&$4.79$&$\cdot 10^{-2}$&\pgfutilensuremath {1.03}&$1.09$&$\cdot 10^{-3}$&\pgfutilensuremath {2.01}\\\bottomrule %
\end {tabular}%

%% file: tables/condition_number_rates_table_2d_3d.tex
\begin {tabular}{r<{\pgfplotstableresetcolortbloverhangright }@{}l<{\pgfplotstableresetcolortbloverhangleft }ccc}%
\toprule \multicolumn {2}{c}{$h$}&$\min _{\delta }\{h^2\kappa _{\delta }(\mathcal {A})\}$&$\max _{\delta }\{h^2\kappa _{\delta }(\mathcal {A})\}$&$\mathrm {mean}_{\delta }\{h^2\kappa _{\delta }(\mathcal {A})\}$\\\midrule %
$1.00$&$\cdot 10^{-1}$&\pgfutilensuremath {1.41}&\pgfutilensuremath {2.14}&\pgfutilensuremath {1.75}\\%
$6.67$&$\cdot 10^{-2}$&\pgfutilensuremath {1.29}&\pgfutilensuremath {2.03}&\pgfutilensuremath {1.59}\\%
$5.00$&$\cdot 10^{-2}$&\pgfutilensuremath {1.26}&\pgfutilensuremath {1.79}&\pgfutilensuremath {1.53}\\%
$3.33$&$\cdot 10^{-2}$&\pgfutilensuremath {1.25}&\pgfutilensuremath {1.67}&\pgfutilensuremath {1.46}\\%
$2.50$&$\cdot 10^{-2}$&\pgfutilensuremath {1.22}&\pgfutilensuremath {1.60}&\pgfutilensuremath {1.45}\\%
$1.67$&$\cdot 10^{-2}$&\pgfutilensuremath {1.22}&\pgfutilensuremath {1.57}&\pgfutilensuremath {1.46}\\\bottomrule %
\end {tabular}%

%% file: tables/condition_number_rates_table_1d_3d.tex
\begin {tabular}{r<{\pgfplotstableresetcolortbloverhangright }@{}l<{\pgfplotstableresetcolortbloverhangleft }ccc}%
\toprule \multicolumn {2}{c}{$h$}&$\min _{\delta }\{h^2\kappa _{\delta }(\mathcal {A})\}$&$\max _{\delta }\{h^2\kappa _{\delta }(\mathcal {A})\}$&$\mathrm {mean}_{\delta }\{h^2\kappa _{\delta }(\mathcal {A})\}$\\\midrule %
$1.00$&$\cdot 10^{-1}$&\pgfutilensuremath {6.11}&\pgfutilensuremath {7.76}&\pgfutilensuremath {6.87}\\%
$6.67$&$\cdot 10^{-2}$&\pgfutilensuremath {6.56}&\pgfutilensuremath {8.13}&\pgfutilensuremath {7.11}\\%
$5.00$&$\cdot 10^{-2}$&\pgfutilensuremath {6.91}&\pgfutilensuremath {7.81}&\pgfutilensuremath {7.41}\\%
$3.33$&$\cdot 10^{-2}$&\pgfutilensuremath {7.86}&\pgfutilensuremath {8.44}&\pgfutilensuremath {8.12}\\%
$2.50$&$\cdot 10^{-2}$&\pgfutilensuremath {7.64}&\pgfutilensuremath {8.64}&\pgfutilensuremath {7.89}\\%
$1.67$&$\cdot 10^{-2}$&\pgfutilensuremath {7.89}&\pgfutilensuremath {8.76}&\pgfutilensuremath {8.09}\\\bottomrule %
\end {tabular}%